\documentclass[a4paper,12pt]{amsart}

\usepackage{fancyhdr}
\usepackage{a4wide}
\usepackage{tikz}
\usepackage{amsmath,amsthm}
\usepackage{float}
\usepackage{color}
\usepackage[noadjust]{cite}
\usepackage{enumitem}
\usepackage{cleveref}
\newcommand{\mylabel}[2]{#2\def\@currentlabel{#2}\label{#1}}
\DeclareMathOperator{\m}{m}

\DeclareMathOperator{\dis}{d}
\DeclareMathOperator{\rr}{r}

\DeclareMathOperator{\h}{h}
\DeclareMathOperator{\rf}{rf}
\DeclareMathOperator{\f}{f}

\DeclareMathOperator{\En}{En}

\DeclareMathOperator{\diag}{diag}

\DeclareMathOperator{\sd}{sd}

\DeclareMathOperator{\Li}{L}

\DeclareMathOperator{\Lel}{LEL}
\DeclareMathOperator{\IE}{IE}
\DeclareMathOperator{\SW}{SW}

\DeclareMathOperator{\F}{\mathcal{F}}
\DeclareMathOperator{\M}{\mathcal{M}}
\DeclareMathOperator{\Ma}{M}
\DeclareMathOperator{\G}{\mathcal{G}}

\DeclareMathOperator{\rrd}{rd}

\newcommand{\changefont}{%
    \fontsize{9}{12}\selectfont}

\newtheorem{thm}{Theorem}[section]
\newtheorem{prop}[thm]{Proposition}
\newtheorem{lem}[thm]{Lemma}
\newtheorem{cor}[thm]{Corollary}
\theoremstyle{remark}

\newtheorem{defn}[thm]{Definition}

\newtheorem*{thm*}{Theorem}

\newif\ifdetails
\detailstrue
\newcommand{\DETAIL}[1]%
{\ifdetails\par\fbox{\begin{minipage}{0.9\linewidth}\textit{Detail:}
      #1\end{minipage}}\par\fi}
\newcommand{\TODO}[1]%
{\ifdetails\par\fbox{\begin{minipage}{0.9\linewidth}\textbf{TODO:}
      #1\end{minipage}}\par\fi}

\newcommand{\old}[1]{{}}

\title{Extremal trees with fixed degree sequence}

\author{Eric O. D. Andriantiana}
\address{Eric O. D. Andriantiana\\
Department of Mathematics (Pure and Applied)\\
Rhodes University, PO Box 94\\
6140 Grahamstown\\
South Africa}
\email{E.Andriantiana@ru.ac.za}

\author{Valisoa Razanajatovo Misanantenaina}
\address{Valisoa Razanajatovo Misanantenaina\\
Department of Mathematical Sciences\\
Stellenbosch University\\
Private Bag X1\\
Matieland 7602\\
South Africa}
\email{valisoa@aims.ac.za}

\author{Stephan Wagner}
\address{Stephan Wagner\\
Department of Mathematics \\
Uppsala Universitet \\
Box 480 \\
751 06 Uppsala \\
Sweden
\and
Department of Mathematical Sciences\\
Stellenbosch University\\
Private Bag X1\\
Matieland 7602\\
South Africa
}
\email{stephan.wagner@math.uu.se,swagner@sun.ac.za}
\thanks{This work was supported by the National Research Foundation of South Africa (grants 96236 and 96310).}


\begin{document}

\begin{abstract}
The greedy tree $\G(D)$ and the $\M$-tree $\M(D)$ are known to be extremal among trees with degree sequence $D$ with respect to various graph invariants. This paper provides a general theorem that covers a large family of invariants for which $\G(D)$ or $\M(D)$ is extremal. Many known results, for example on the Wiener index, the number of subtrees, the number of independent subsets and the number of matchings follow as corollaries, as do some new results on invariants such as the number of rooted spanning forests, the incidence energy and the solvability. 
We also extend our results on trees with fixed degree sequence $D$ to the set of trees whose degree sequence is majorised by a given sequence $D$, which also has a number of applications.
\end{abstract}

\maketitle

\section{introduction}
 In the context of chemical graph theory, (molecular) graphs are used to model molecules: the vertices of the graph represent the atoms of the molecule, and the edges correspond to the chemical bonds between atoms. Thus vertex degrees amount to valencies of atoms. If the entries of  $D=(d_1,d_2,\dots,d_n)$ are possible valencies of atoms, then the set 
 of all graphs with degree sequence $D$ contains the molecular graphs of all possible isomers of a certain molecule. This is one of the motivations to study graphs with a given degree sequence. In this paper, we will be particularly concerned with trees whose degree sequence is given.

The \emph{greedy tree} $\G(D)$ with degree sequence $D$, formally defined in the next section, is a tree that can be constructed by starting with the largest degree vertex and always assigning the largest available degree to a neighbour of the vertex with the largest degree whose neighbour degrees are not yet fully specified. On the other hand, large degrees and small degrees alternate in the \emph{$\M$-tree} $\M(D)$ with degree sequence $D$, see  Definition~\ref{Def:AltGreed}. 

The greedy tree $\G(D)$ and the $\M$-tree $\M(D)$ are known to be extremal for a number of graph invariants. Let us list a few examples.
\begin{itemize}
\item The sum $W(G)$ of the distances $\dis(u,v)$ between all (unordered) pairs of vertices $\{u, v\}$, better known as the Wiener index, is among the most popular graph invariants. It was shown by Wang in \cite{WangWi} and by Zhang et al.~in \cite{Zhang} that, among all trees with degree sequence $D$, $\G(D)$ has the minimum Wiener index.

\item  Schmuck et al.~\cite{Schmuck} showed a more general statement: if $W(G)$ is replaced by a Wiener-like invariant
\[
W_f(G)=\sum_{\{u,v\}\subseteq V(G)}f(\dis(u,v))
\]
for a nonnegative and nondecreasing function $f$, then $\G(D)$ still attains the minimum value among trees with degree sequence $D$. Likewise, if $f$ is a nonnegative and nonincreasing function, then $\G(D)$ attains the maximum value. In particular, this includes a well-studied invariant known as the Harary index, which corresponds to $f(x)=1/x$ (see \cite{Wagnerhar}).
\item  The greedy tree $\G(D)$ is extremal with respect to several other invariants including:
\begin{itemize}
\item the number of subtrees (and the number of antichains when rooted trees are considered) \cite{Andriantiana1, Zhang1},
\item the spectral radius \cite{biyi} and the Laplacian spectral radius \cite{Zhang2}, 
\item the spectral moments \cite{Andriantiana2, Shuchao}. See also \cite{cgjz} for a generalisation.
\end{itemize}

\item The total number of independent subsets is also known as the Merrifield-Simmons index of a graph. It is proven in \cite{Dadah1} that the Merrifield-Simmons index is maximised by $\M(D)$. The total number of matchings is also called Hosoya index. It is minimised by $\M(D)$ (see \cite{Dadah1}).

\item The energy of a graph $G$ is defined as
$
\En(G)=\sum_{i}|\lambda_i|,
$
where the $\lambda_i$s are the eigenvalues of the adjacency matrix of $G$. For trees, the energy is connected to matchings by the well-known Coulson formula \cite{Gutman3}: let $\m(G,k)$ be the number of matchings of cardinality $k$ in a graph $G$. For every tree $T$,
\begin{equation}
\label{Eq:EnTree}
\En(T)=\frac{2}{\pi}\int_{0}^{\infty}\frac{dx}{x^2}\log\sum_{k\geq 0}\m(T,k)x^{2k}.
\end{equation}
For a positive real number $x$, define 
\[
\Ma(G,x)=\sum_{k\geq 0}\m(G,k)x^k.
\]
It is shown in \cite{Dadah1} that, among all trees with degree sequence $D$, $\M(D)$ is the unique tree with minimum $\Ma(.,x)$ for any $x>0$. Thus, it also has the minimum energy.
\end{itemize}

The long list of graph invariants for which $\G(D)$ or $\M(D)$ is extremal leads us to the following natural question: what condition(s) does a graph invariant need to satisfy for $\G(D)$ or $\M(D)$ to be extremal among trees with degree sequence $D$? In this paper, we provide sufficient conditions that cover a number of examples, both old and new. The idea is inspired by an exchange lemma used to study bounded degree trees in \cite{Wagn}, and trees with given degree sequence in \cite{Dadah1}.

A sequence $B=(b_1,b_2,\dots,b_n)$ is said to majorise $D$ if $b_1 + \dots + b_n = d_1 + \dots + d_n$ and $d_1+\dots+d_k\leq b_1+\dots+b_k$ for every $k$ with $1\leq k< n.$ We then write $D\lhd B$. In Section~\ref{bounded}, we will consider the set of all trees whose degree sequence is majorised by $D$. Several natural sets of trees satisfy majorisation properties.  This includes  the set of all $n$-vertex trees, $n$-vertex trees with bounded degrees, and trees of order $n$ with fixed number of leaves. There are several instances of graph invariants for which the extremal trees with degree sequence $D$ are also extremal among those whose degree sequence is majorised by $D$. We will see many such examples at the end of this paper.

We need to introduce some notation to describe our results.
Let $T$ be a rooted tree. The root of $T$ is denoted by $\rr(T)$.
We write $T=[T_1,T_2,\dots, T_k]$ if the rooted trees $T_1,T_2,\dots, T_k$ are the connected components of $T-\rr(T)$, where the roots of $T_1,T_2,\dots,T_k$ are the neighbours of $\rr(T)$.

Let $v$ and $w$ be two different leaves of a tree $H$. The tree obtained by merging the root of $[L_1,L_2,\dots,L_k]$ with $v$ and the root of $[R_1,R_2,\dots,R_{\ell}]$ with $w$, as seen in Figure~\ref{cha5.fig1}, is denoted by $[L_1,L_2,\dots,L_k]vHw[R_1,R_2,\dots,R_{\ell}]$. 
 
\begin{figure}[H]
\centering
\begin{tikzpicture}[scale=0.7]
\node[fill=black,circle,inner sep=1pt]  at (0,0) {};
\node[fill=black,circle,inner sep=1pt]  at (0,2.21) {};
\node[fill=black,circle,inner sep=1pt]  at (0,-3) {};
\node[fill=black,circle,inner sep=1pt]  at (4,0) {};
\node[fill=black,circle,inner sep=1pt]  at (4,2.21) {};
\node[fill=black,circle,inner sep=1pt]  at (4,-3) {};
\node[fill=black,circle,inner sep=1pt]  at (1,0) {};
\node[fill=black,circle,inner sep=1pt]  at (3,0) {};
\draw (0,0)--(-1,1)--(-1,-1)-- (0,0);
\draw (0,-3)--(-1,-2)--(-1,-4)-- (0,-3);
\draw (0,2.2)--(-1,3.2)--(-1,1.2)-- (0,2.2);
\draw (4,0)--(5,1)--(5,-1)--(4,0);
\draw (4,-3)--(5,-2)--(5,-4)--(4,-3);
\draw (4,2.2)--(5,3.2)--(5,1.2)--(4,2.2);
\draw (0,0)--(1,0)--(0,2.2);
\draw (1,0)--(0,-3);
\draw (4,0)--(3,0)--(4,2.2);
\draw (3,0)--(4,-3);
\draw[dashed] (2,0) circle (1);
\node at (2,1.2) {$H$};
\node at (2.8,0.2) {$w$};
\node at (1.2,0.2) {$v$};
\node at (-0.5,2.2) {$L_1$};
\node at (-0.5,0) {$L_2$};
\node at (-0.5,-3) {$L_{k}$};
\node at (4.5,2.2) {$R_1$};
\node at (4.5,0) {$R_2$};
\node at (4.5,-3) {$R_{\ell}$};
\node[fill=black,circle,inner sep=0.5pt]  () at (0,-1) {};
\node[fill=black,circle,inner sep=0.5pt]  () at (0,-1.6) {};
\node[fill=black,circle,inner sep=0.5pt]  () at (0,-2.2) {};
\node[fill=black,circle,inner sep=0.5pt]  () at (4,-1) {};
\node[fill=black,circle,inner sep=0.5pt]  () at (4,-1.6) {};
\node[fill=black,circle,inner sep=0.5pt]  () at (4,-2.2) {};
\end{tikzpicture}

\caption{The tree $[L_1,L_2,\dots,L_k]vHw[R_1,R_2,\dots,R_{\ell}]$.}
\label{cha5.fig1}
\end{figure}
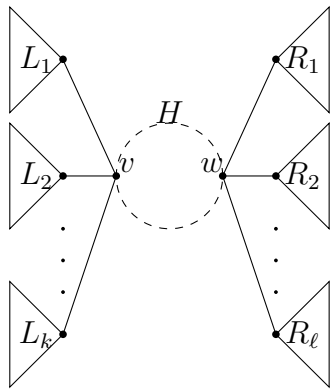 

Let $\rho(T)$ be a quantity associated to a rooted tree $T$, which satisfies a recursive relation of the form

\begin{equation}
\rho([T_1,T_2,\dots,T_k])=f_{\rho}(\rho(T_1),\rho(T_2),\dots,\rho(T_k))\label{cha5.eq1}
\end{equation}
for a symmetric function $f_{\rho}$, i.e.,~the value of $f_{\rho}$ is invariant under permutation of the branches. We call $f_{\rho}$ the \emph{recurrence rule} for $\rho$. In our examples, the recurrence rule can always be expressed as a function of sums or products over all branches.

\begin{defn}\label{cha5.defn1}
We say that a tree $T$ is \emph{$\rho$-exchange-extremal} if, whenever we can decompose $T$ as

\[T=[L_1,L_2,\dots,L_k]vHw[R_1,R_2,\dots,R_{\ell}]\]
for some $H$, then we have $k \geq \ell$ and 
\[\min\{\rho(L_1),\rho(L_2),\dots, \rho(L_{k})\} \geq \max\{\rho(R_1),\rho(R_2),\dots,\rho(R_{\ell})\}\]
or $k \leq \ell$ and
\[\max\{\rho(L_1),\rho(L_2),\dots, \rho(L_{k})\} \leq \min\{\rho(R_1),\rho(R_2),\dots,\rho(R_{\ell})\}.\]
\end{defn}

We will prove the following theorem:
\begin{thm*}
Let $\rho$ be an invariant of rooted trees that satisfies the recurrence relation \eqref{cha5.eq1}. Suppose that a tree $T$ with degree sequence $D$ is $\rho$-exchange-extremal. 
\begin{itemize}
\item If $f_{\rho}$ is increasing (in each variable and under addition of further variables) and the minimum of $\rho$ is attained by the single-vertex tree, then $T$ has to be $\G(D)$.
\item If $f_{\rho}$ is decreasing (in each variable and under addition of further variables) and the maximum of $\rho$  is attained by the single-vertex tree, then $T$ has to be $\M(D)$.
\end{itemize}
\end{thm*}

It turns out that extremality with respect to many different invariants implies $\rho$-exchange-extremality for a suitable choice of $\rho$. This includes some of the previously mentioned instances (Wiener index, number of subtrees, Merrifield-Simmons index, Hosoya index, etc.), but we will also discuss examples that have not been considered before. 

The rest of this paper is organised as follows. Section~\ref{recinc} treats the case of increasing $f_{\rho}$, where the greedy tree $\G(D)$ is extremal. The case of decreasing $f_{\rho}$ is considered in Section~\ref{recdec}, it leads to the situation where the $\M$-tree $\M(D)$ is extremal. Various special cases are listed at the end of each section. This includes many known results, but we also obtain several new results, specifically for the number of rooted spanning forests (related to the coefficients of the Laplacian characteristic polynomial), the incidence energy, and an invariant called the solvability. We also settle an open question on the Steiner Wiener index. Furthermore, in Section~\ref{bounded}, we compare trees with different degree sequences and prove extremality results for the set of all trees whose degree sequence is majorised by a fixed sequence $D$.

The following technical terms will be needed.
The \emph{height} $\h(T)$ of a rooted tree $T$ is the greatest distance of a vertex from $\rr(T)$.
A subgraph $B$ of a tree $T$ is called a \emph{complete branch} of $T$ if there is an edge $e$ such that $B$ is one of the components of $T-e$, see Figure~\ref{cha5.fig}. We denote by $\rrd(B)$ the degree of the root $\rr(B)$ (the end of $e$ that belongs to $B$) as a vertex of $B$.

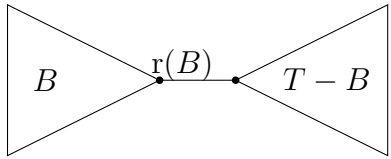
\begin{figure}[H]
\centering
\begin{tikzpicture}[scale=1]
\node[fill=black,circle,inner sep=1pt] (t1) at (0,0) {};
\node[fill=black,circle,inner sep=1pt] (t3) at (1,0) {};
\draw (0,0)--(-2,1)--(-2,-1)-- (0,0)-- (1,0);
\draw (1,0)--(3,1)--(3,-1)-- (1,0);
\node at (0.3,0.2) {$\rr(B)$};
\node at (-1.5,0) {$B$};
\node at (2.2,0) {$T-B$};
\end{tikzpicture}

\caption{Complete branches $B$ and $T-B$ of $T$.}
\label{cha5.fig}
\end{figure}

\section{Increasing recurrence rule $f_{\rho}$}\label{recinc}

\begin{defn} \cite{Schmuck}
Given a degree sequence of a tree $D$, the greedy tree, denoted $\G(D)$, is constructed by the following ``greedy algorithm'':
	\begin{enumerate}
	\item[(i)] Label the vertex with the largest degree $v$ (the root).
	\item[(ii)] Label the neighbours of $v$ as $v_1,v_2,\dots,v_{d(v)}$, and assign the largest degrees available to them such that $d(v_1)\geq d(v_2)\geq \cdots\geq d(v_{d(v)})$.
	\item[(iii)] Label the neighbours of $v_1$ (except $v$) as $v_{1,1},v_{1,2},\dots,v_{1,d(v_1)-1}$, and assign the largest degrees available to them such that $d(v_{1,1})\geq d(v_{1,2})\geq \cdots\geq d(v_{1,d(v_1)-1})$. Then do the same for $v_2,v_3,\dots,v_{d(v)}$.
	\item[(iv)] Repeat (ii) and (iii) for all the newly labelled vertices. Always start with the neighbours of the labelled vertex with the largest degree whose neighbours are not labelled yet.
	\end{enumerate}
\end{defn}
Figure~\ref{Fig:Greedy} shows the greedy tree $\G(D)$ for $D=(4,4,3,3,3,3,3,3,2,2,2,1,1,\dots,1)$.
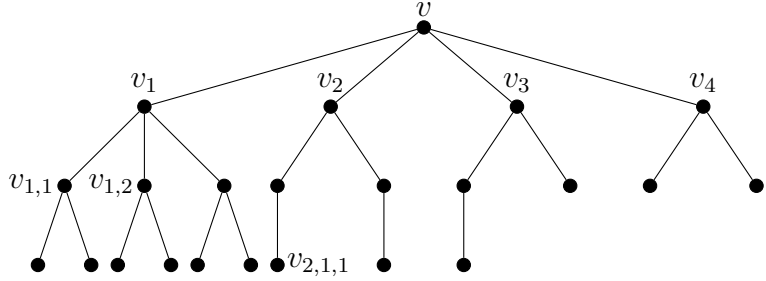
\begin{figure}[h!]
\centering
\begin{tikzpicture}[scale=0.7]
\tikzstyle{every node}=[draw,circle,fill=black,minimum size=5pt,
                            inner sep=0pt]
\node [label=above:$v$]{}[sibling distance=35mm]
    child {  node [label=above:$v_1$]{}  [sibling distance=15mm]
    					child {node[label=left:$v_{1,1}$]{}[sibling distance=10mm]
    						child{node{}} 
    						child{node{}}
    					}
    					child{node[label=left:$v_{1,2}$]{}[sibling distance=10mm]
						child{node{}} 
    						child{node{}}    					
    					} 
    					child{node{}[sibling distance=10mm]
						child{node{}} 
    						child{node{}}    					
    					}}
    	child {  node [label=above:$v_2$]{} [sibling distance=20mm]
                    child { node {} 
                    		child{node[label=right:$v_{2,1,1}$]{}}}
                    child{node{}
						child{node{}}                    
                    }
             }
    child {  node [label=above:$v_3$]{} [sibling distance=20mm]
                    child { node {} 
                    			child{node{}}}
                    child{node{}}
             }
    child {  node [label=above:$v_4$]{} [sibling distance=20mm]
                    child { node {} }
                    child{node{}}
             }  ;
\end{tikzpicture}
\caption{The greedy tree $\G(4,4,3,3,3,3,3,3,2,2,2,1,1,\dots,1)$.}
\label{Fig:Greedy}
\end{figure}

We will see in this section that the greedy tree is the unique $\rho$-exchange-extremal tree if $\rho$ satisfies the following conditions:
\begin{enumerate}[label=\textbf{I.\arabic*}]
\item \label{cond1} the recurrence relation \eqref{cha5.eq1} holds, for some symmetric recurrence rule $f_{\rho}$,
\item  the function $f_{\rho}$ is strictly increasing (strictly increasing in each single variable and strictly increasing under addition of further variables),\label{cond2}
\item $\rho(\bullet)<\rho(B),\label{cha5.eq2}$
for all rooted trees $B$ with $|V(B)|>1$, where $\bullet$ denotes a single vertex tree. \label{cond3}
\end{enumerate}

\subsection{Special case: the Wiener index}

Let us first consider the special case where $\rho(T)$ is the number of vertices of $T$, i.e., $\rho(T)=\rho_0(T)=|V(T)|$. We have

\begin{align*}
\rho_0([T_1,\dots,T_k])=f_{\rho_0}(\rho_0(T_1),\dots,\rho_0(T_k))=1+\sum_{i=1}^k\rho_0(T_i).
\end{align*}

The recurrence rule $f_{\rho_0}$ is indeed symmetric and increasing with respect to each of its variables and under addition of further variables. Clearly $\rho_0(\bullet)=1$ is the minimum among all rooted trees. So conditions~\ref{cond1} to~\ref{cond3} are all satisfied.

Now recall that the Wiener index is the sum of all distances between pairs of vertices:

\begin{equation*}
W(T)=\sum_{u,v \in V(T)}\dis(u,v).
\end{equation*}

An important alternative formula for the Wiener index for trees is (see \cite{dobrynin})

\begin{equation}
W(T)=\sum_{uv \in E(T)}\rho_0(T_u) \rho_0(T_v), \label{cha5.wie4}
\end{equation}
where $T_u$ and $T_v$ are the components of $T-uv$ containing $u$ and $v$ respectively. We use this formula to show that minimality with respect to the Wiener index implies $\rho_0$-exchange-extremality.

\begin{lem}[cf.~\cite{szekely}]\label{cha5.wie6}
Let $T$ be a tree with degree sequence $D$ for which $W(T)$ attains its minimum. Then, for any two disjoint complete branches $A=[A_1,\dots,A_k]$ and $B=[B_1,\dots,B_{\ell}]$ in $T$, we have 
\begin{itemize}
\item either $k \geq \ell$ and $\min\{\rho_0(A_1),\dots,\rho_0(A_{k})\} \geq \max\{\rho_0(B_1),\dots,\rho_0(B_{\ell})\}$,

\item or $k \leq \ell$ and $\max\{\rho_0(A_1),\dots, \rho_0(A_{k})\} \leq \min\{\rho_0(B_1),\dots,\rho_0(B_{\ell})\}.$
\end{itemize}

\end{lem}

\begin{proof}

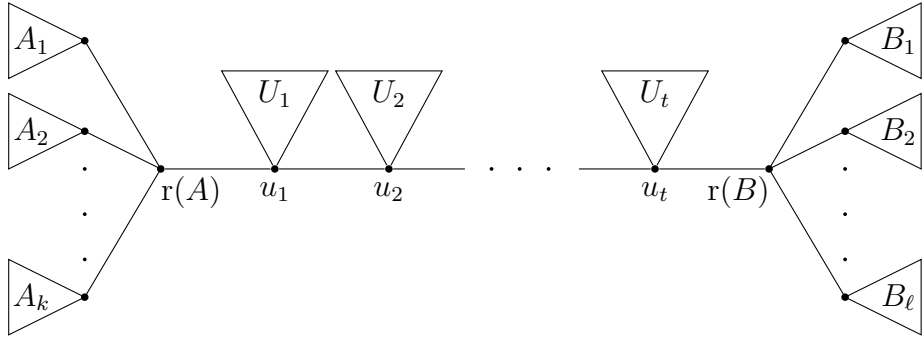
\begin{figure}[H]
\centering
\begin{tikzpicture}[scale=1]
\node[fill=black,circle,inner sep=1pt]  at (0,0.5) {};
\node[fill=black,circle,inner sep=1pt]  at (0,1.7) {};
\node[fill=black,circle,inner sep=1pt]  at (0,-1.7) {};
\node[fill=black,circle,inner sep=1pt]  at (10,0.5) {};
\node[fill=black,circle,inner sep=1pt]  at (10,1.7) {};
\node[fill=black,circle,inner sep=1pt]  at (10,-1.7) {};
\node[fill=black,circle,inner sep=1pt]  (t1) at (1,0) {};
\node[fill=black,circle,inner sep=1pt]  (t2) at (9,0) {};
\node[fill=black,circle,inner sep=1pt]  (t3) at (2.5,0) {};
\node[fill=black,circle,inner sep=1pt]  (t4) at (4,0) {};
\node[fill=black,circle,inner sep=1pt]  (t5) at (7.5,0) {};
\node[fill=black,circle,inner sep=0.5pt]  () at (0,0) {};
\node[fill=black,circle,inner sep=0.5pt]  () at (0,-0.6) {};
\node[fill=black,circle,inner sep=0.5pt]  () at (0,-1.2) {};
\node[fill=black,circle,inner sep=0.5pt]  () at (10,0) {};
\node[fill=black,circle,inner sep=0.5pt]  () at (10,-0.6) {};
\node[fill=black,circle,inner sep=0.5pt]  () at (10,-1.2) {};
\node[fill=black,circle,inner sep=0.5pt]  () at (5.35,0) {};
\node[fill=black,circle,inner sep=0.5pt]  () at (5.72,0) {};
\node[fill=black,circle,inner sep=0.5pt]  () at (6.1,0) {};
\draw (t1)--(t3)--(t4);
\draw (t5)--(t2);
\draw (t4)--(5,0);
\draw (t5)--(6.5,0);
\draw (t3)--(1.8,1.3)--(3.2,1.3)-- (t3);
\draw (t4)--(3.3,1.3)--(4.7,1.3)-- (t4);
\draw (t5)--(6.8,1.3)--(8.2,1.3)-- (t5);
\draw (0,0.5)--(-1,1)--(-1,0)-- (0,0.5);
\draw (0,-1.7)--(-1,-2.2)--(-1,-1.2)-- (0,-1.7);
\draw (0,1.7)--(-1,2.2)--(-1,1.2)-- (0,1.7);
\draw (10,0.5)--(11,0)--(11,1)--(10,0.5);
\draw (10,-1.7)--(11,-2.2)--(11,-1.2)--(10,-1.7);
\draw (10,1.7)--(11,2.2)--(11,1.2)--(10,1.7);
\draw (0,0.5)--(1,0)--(0,1.7);
\draw (1,0)--(0,-1.7);
\draw (10,0.5)--(9,0)--(10,1.7);
\draw (9,0)--(10,-1.7);
\node at (2.5,-0.3) {$u_1$};
\node at (4,-0.3) {$u_2$};
\node at (7.5,-0.3) {$u_{t}$};
\node at (2.5,1) {$U_1$};
\node at (4,1) {$U_2$};
\node at (7.5,1) {$U_{t}$};
\node at (8.6,-0.3) {$\rr(B)$};
\node at (1.4,-0.3) {$\rr(A)$};
\node at (-0.7,1.7) {$A_1$};
\node at (-0.7,0.5) {$A_2$};
\node at (-0.7,-1.7) {$A_{k}$};
\node at (10.7,1.7) {$B_1$};
\node at (10.7,0.5) {$B_2$};
\node at (10.7,-1.7) {$B_{\ell}$};
\end{tikzpicture}

\caption{Decomposition of $T$ in the proof of Lemma~\ref{cha5.wie6}.}
\label{cha5.fig5}
\end{figure}
Let $P=\rr(A) u_1\dots u_t\rr(B)$ be the unique path between $\rr(A)$ and $\rr(B)$. To simplify notation, we put $\alpha=\rho_0(A)=1+\sum_{i=1}^k\rho_0(A_i)$ and $\beta=\rho_0(B)=1+\sum_{i=1}^{\ell}\rho_0(B_i)$.

For $1\leq j\leq t$, let $U_j$ be the component containing $u_j$ when we remove all the edges of the path $P$, and set $z_j=\rho_0(U_j)$, $p_j=z_1+\cdots+z_j$, $q_j=z_{t}+\cdots+z_{t-j+1}$ and $p_0=q_0=0$.

Using the formula for the Wiener index in \eqref{cha5.wie4}, we have:

\begin{align*}
W(T)&=\sum_{j=0}^{t}(\alpha+p_j)(q_{t-j}+\beta)+C_T,
\end{align*}
where $C_T$ is the contribution of edges in $U_1,\dots,U_t,A$ and $B$.

With $q=q_1+\dots+q_t$ and $p=p_1+\dots+p_t$, this can be rewritten as

\[W(T)=\alpha q+\beta p+\alpha \beta (t+1) + \sum_{j=0}^{t} p_jq_{t-j}+C_T.\]

The last two terms $\sum_{j=0}^tp_jq_{t-j}+C_T$ are invariant under any rearrangements of the $A_i$s and $B_i$s. Note that the degree sequence remains the same if all the $A_i$s are switched with all the $B_i$s (interchanging the degrees of $\rr(A)$ and $\rr(B)$), if $A_1,A_2,\ldots,A_k,B_1,B_2,\ldots,B_{\ell}$ are permuted in an arbitrary way, or both. 
Hence, if $W(T)$ is minimal, then the expression $\alpha q+\beta p+\alpha \beta (t+1)$ must have its minimum value under these operations. Note that the sum $\Sigma=\alpha + \beta$ remains constant under any of these rearrangements, and we have
\[\alpha q+ \beta p + \alpha \beta (t+1) = \alpha q + (\Sigma - \alpha)p+ \alpha(\Sigma - \alpha)(t+1),\]

which is a strictly concave function of $\alpha$. Thus the minimum is attained if $\alpha$ is either as small or as large as possible. For this we must have 
\[k\leq \ell \quad\text{and}\, \max\{\rho_0(A_i): 1 \leq i \leq k\} \leq \min\{\rho_0(B_i): 1 \leq i \leq \ell\}\]
or 
\[k \geq \ell \quad \text{and}\, 
\min\{\rho_0(A_i): 1 \leq i \leq k\} \geq \max\{\rho_0(B_i): 1 \leq i \leq \ell\}.\]
In words, the largest possible branches (and as many of them as possible) have to be gathered in one place.
This is precisely the statement of the lemma.

\end{proof}

We know now that the tree which minimises the Wiener index is $\rho_0$-exchange-extremal. Moreover, $\rho_0$-exchange-extremality is equivalent to the semi-regular property defined in Definition~4 of \cite{Schmuck}, where it is proved that such a tree is greedy. Thus we have the following observation, which will be the basis for the proof of the main theorem of this section.
\begin{thm}[cf.~\cite{Schmuck,szekely}]\label{cha5.thm0}
If $T$ is a $\rho_0$-exchange-extremal tree, then $T$ is a greedy tree.
\end{thm}

\subsection{Main result}

We require one more lemma before stating the main result. This lemma establishes a connection between $\rho_0$ and other invariants $\rho$.

\begin{lem}\label{lemthm}
Let $\rho$ be an invariant of rooted trees that satisfies \ref{cond1},\ref{cond2} and \ref{cond3}, and let $T$ be a $\rho$-exchange-extremal tree. For any two disjoint complete branches $A$ and $B$ in $T$, $\rho(A) \geq \rho(B)$ if and only if $\rho_0(A) \geq \rho_0(B)$. In particular, $\rho(A)=\rho(B)$ if and only if $\rho_0(A)=\rho_0(B)$.
\end{lem}
\begin{proof}

We reason by induction on the heights of $A$ and $B$, specifically $\max\{\h(A),\h(B)\}$. If $\max\{\h(A),\h(B)\}=0$, then $\h(A)=\h(B)=0$, so $A$ and $B$ are both single vertices and the claim holds trivially. Assume that the statement holds whenever $\max\{\h(A),\h(B)\} \leq t$, for some $t \geq 0$. 
Now, consider two complete branches $A$ and $B$ in $T$ such that $\max\{\h(A),\h(B)\}=t+1$ and $\rho(A)\geq \rho(B)$. If $A$ only has a single vertex, then so does $B$ by condition \ref{cond3}. If $B$ only has a single vertex, then $\rho_0(A) \geq \rho_0(B)$ is trivially satisfied. In both cases, there is nothing left to show. So we can now assume that $A$ and $B$ have more than one vertex and can thus be decomposed. Let $A=[A_1,\dots,A_k]$ and $B=[B_1,\dots,B_{\ell}]$ for some $k,\ell\geq 1$. Since $T$ is a $\rho$-exchange-extremal tree, we are left with two possibilities. If $k \geq \ell$, and 
$\min\{\rho(A_i):1\leq i \leq k \} \geq \max\{\rho(B_i):1 \leq i \leq \ell\},$ then by the induction hypothesis, we have 
$
\min\{\rho_0(A_i): 1 \leq i \leq k\} \geq \max \{\rho_0(B_i): 1 \leq i \leq \ell\}. 
$
Thus $\rho_0(A)\geq \rho_0(B)$, because $f_{\rho_0}$ is increasing. 

On the other hand, if $k\leq \ell$ and 
$\max\{\rho(A_i):1\leq i \leq k \} \leq \min\{\rho(B_i):1 \leq i \leq \ell\},$ then $\rho(A) \leq \rho(B)$ since $f_{\rho}$ is increasing, so $\rho(A)=\rho(B)$. For this to hold, we must have
$k=\ell$ and $\rho(A_1)=\cdots=\rho(A_k)=\rho(B_1)=\cdots=\rho(B_{\ell})$. By the induction hypothesis, this implies $\rho_0(A_1)=\cdots=\rho_0(A_k)=\rho_0(B_1)=\cdots=\rho_0(B_{\ell})$, so $\rho_0(A)=\rho_0(B)$.

Thus we have shown that $\rho(A) \geq \rho(B)$ implies $\rho_0(A) \geq \rho_0(B)$, and the proof of the converse is analogous.
\end{proof}

Now we are ready for the main theorem.
\begin{thm}\label{cha5.thm}
Let $\rho$ be an invariant of rooted trees that satisfies \ref{cond1},\ref{cond2} and \ref{cond3}. If a tree $T$ is $\rho$-exchange-extremal, then $T$ is a greedy tree.
\end{thm}
\begin{proof}
By Lemma~\ref{lemthm}, the tree $T$ is also $\rho_0$-exchange-extremal. Theorem \ref{cha5.thm0} gives us the desired result. 
\end{proof}

In the remainder of this section, we describe different applications of Theorem \ref{cha5.thm}.

\subsection{The number of subtrees}\label{subsec:subtrees}

As a first example, we consider the number of subtrees (here, a subtree is any induced subgraph that is connected and thus again a tree). We will be able to show that there is a suitable function $\rho$ such that a tree with the greatest possible number of subtrees among all trees with the same degree sequence is necessarily $\rho$-exchange-extremal. To this end, we need a technical lemma, which will also be useful later.

\begin{lem}
\label{Lem:RearrIneq}
Let $x_1\geq \dots\geq x_{n} \geq 1$ and $a\geq b$ be positive real numbers,  $n\geq k \geq n/2$ and $\mathbb{S}_n$ the set of permutations of $1,\dots,n$. Then
\[a\prod_{i=1}^{k}x_{i}+b\prod_{i=k+1}^{n}x_{i}=\max\left\{a\prod_{i=1}^{\ell}x_{\sigma(i)}+b\prod_{i=\ell+1}^{n}x_{\sigma(i)}: \sigma\in\mathbb{S}_n\text{ and } \ell\in\{n-k,k\}\right\}.\]
Moreover, we have uniqueness: if
\[a\prod_{i=1}^{k}x_{i}+b\prod_{i=k+1}^{n}x_{i}=a\prod_{i=1}^{\ell}x_{\sigma(i)}+b\prod_{i=\ell+1}^{n}x_{\sigma(i)},\]
$a>b$ and the $x_i$s are not all equal to $1$, then $\ell = k$ and $x_{\sigma(1)},\ldots,x_{\sigma(k)}$ are a permutation of $x_1,\ldots,x_k$. If $a=b$, then either the same holds, or $\ell = n-k$ and $x_{\sigma(1)},\ldots,x_{\sigma(n-k)}$ are a permutation of $x_{k+1},\ldots,x_n$.
\end{lem}

In words, Lemma~\ref{Lem:RearrIneq} says that the maximum value is reached when we assemble the largest elements together in the same product.

\begin{proof}
Set $g(x) = ax + b x^{-1} \prod_{i=1}^n x_i$, and note that 
\[a\prod_{i=1}^{\ell}x_{\sigma(i)}+b\prod_{i=\ell+1}^{n}x_{\sigma(i)} = g\Big(\prod_{i=1}^{\ell}x_{\sigma(i)} \Big).\]
Since $g''(x) = 2bx^{-3} \prod_{i=1}^n x_i > 0$, $g$ is strictly convex, so it can only attain its maximum value when $x = \prod_{i=1}^{i= \ell}x_{\sigma(i)}$ is either largest or smallest. The largest possible value of $x$ is clearly $L=\prod_{i=1}^k x_i$, the smallest $S=\prod_{i=k+1}^n x_i$. Comparing the two possibilities, we see that
\[ g(L)-g(S)
=a \prod_{i=1}^k x_i + b\prod_{i=k+1}^{n}x_{i} - a \prod_{i=k+1}^{n}x_{i} - b \prod_{i=1}^k x_i
=(a-b)\left(\prod_{i=1}^k x_i-\prod_{i=k+1}^n x_i\right)\geq 0  \]
since $a \geq b$. We have strict inequality if $a >b$ unless $k=\ell$ and all $x_i$s are equal (in which case $x_{\sigma(1)},\ldots,x_{\sigma(k)}$ are still a permutation of $x_1,\ldots,x_k$) or all $x_i$s are equal to $1$. The statement of the lemma follows.
\end{proof}

The number of subtrees of a tree $T$ is denoted by $\eta(T)$, and the number of subtrees of $T$ containing the vertices $v_1,\dots,v_{\ell}$ is denoted by $\eta(T,v_1,\dots,v_{\ell})$. In the following lemma, we establish the necessary property to apply our main theorem.

\begin{lem}\label{lemsub}
Let $T$ be a tree with degree sequence $D$ for which $\eta(T)$ attains its maximum. For any two disjoint complete branches $A=[A_1,\dots,A_k]$ and $B=[B_1,\dots,B_{\ell}]$ in $T$, we have 

\begin{itemize}
\item either $k \geq \ell$ and \[\min\{\eta(A_1,\rr(A_1)),\dots,\eta(A_{k},\rr(A_k))\} \geq \max\{\eta(B_1,\rr(B_1)),\dots,\eta(B_{\ell},\rr(B_{\ell}))\},\]

\item or $k \leq \ell$ and \[\max\{\eta(A_1,\rr(A_1)),\dots, \eta(A_k,\rr(A_{k}))\} \leq \min\{\eta(B_1,\rr(B_1)),\dots,\eta(B_{\ell},\rr(B_{\ell}))\}.\]
\end{itemize}

In other words, $T$ is $\rho$-exchange-extremal with $\rho(T)= \eta(T,\rr(T))$.
\end{lem}
\begin{proof}

Consider a decomposition of $T$ as $[A_1,\dots,A_k]vHw[B_1,\dots,B_{\ell}]$, for some tree $H$. Each subtree of $T$ has to belong to one of the following types:

\begin{itemize}
\item Subtrees which contain neither $v$ nor $w$. Each subtree of this type is either in $H-u-v$ or in one of the $A_i$s or $B_i$s. Thus, there are  $\eta(H-v-w)+\sum_{i=1}^k\eta(A_i)+\sum_{i=1}^{\ell}\eta(B_i)$ of them.
\item Subtrees which contain $w$, but not $v$. A subtree of this type consists of a part in $H-v$ and another part in $[B_1,\dots,B_{\ell}]$, each containing $w$.  Their number is 
\[\eta(H-v,w)\prod_{i=1}^{\ell}\left(1+\eta(B_i,\rr(B_i))\right).\]
\item Subtrees which contain $v$, but not $w$. By similar reasoning as in the second case, the number of these subtrees is
\[\eta(H-w,v)\prod_{i=1}^k\left(1+\eta(A_i,\rr(A_i))\right).\]
\item Subtrees which contain both $v$ and $w$. This corresponds to subtrees of $H$ which contain both $v$ and $w$, to which induced subtrees of the $A_i$s and $B_i$s containing the respective roots can be attached. Thus, there are
\[\eta(H,v,w)\prod_{i=1}^{\ell}\left(1+\eta(B_i,\rr(B_i))\right)\prod_{i=1}^k\left(1+\eta(A_i,\rr(A_i))\right)\]
of them.
\end{itemize}

In total we have

\begin{align*}
\eta(T)
&=\alpha_T+\beta_T,
\end{align*}
where
\begin{align*}
\alpha_T&= \eta(H-v,w)\prod_{i=1}^{\ell}\left(1+\eta(B_i,\rr(B_i))\right)+\eta(H-w,v)\prod_{i=1}^k\left(1+\eta(A_i,\rr(A_i))\right),\\
\beta_T&=\eta(H-v-w)+\sum_{i=1}^k\eta(A_i)+\sum_{i=1}^{\ell}\eta(B_i)\\
&\quad+\eta(H,v,w)\prod_{i=1}^{\ell}\left(1+\eta(B_i,\rr(B_i))\right)\prod_{i=1}^k\left(1+\eta(A_i,\rr(A_i))\right).
\end{align*}
As in the proof of Lemma~\ref{cha5.wie6}, this must be maximal under all permutations of the $A_i$s and $B_i$s, possibly also switching the degrees of $v$ and $w$. Note here that $\beta_T$ is not affected by any of these rearrangements. So the maximality of $\eta(T)$ depends only on $\alpha_T$.

Since all the quantities involved are positive and the factors $1 + \eta(A_i,\rr(A_i))$ and $1 + \eta(B_i,\rr(B_i))$ are even greater than $1$, Lemma~\ref{Lem:RearrIneq} applies.
Thus we obtain the maximum of $\alpha_T$ if either $\eta(H-w,v) \geq \eta(H-v,w)$, $k \geq \ell$ and 
\[\min\{\eta(A_1,\rr(A_1)),\dots,\eta(A_{k},\rr(A_k))\} \geq \max\{\eta(B_1,\rr(B_1)),\dots,\eta(B_{\ell},\rr(B_{\ell}))\},\]

or $\eta(H-w,v) \leq \eta(H-v,w)$, $k \leq \ell$ and
\[\max\{\eta(A_1,\rr(A_1)),\dots, \eta(A_k,\rr(A_{k}))\} \leq \min\{\eta(B_1,\rr(B_1)),\dots,\eta(B_{\ell},\rr(B_{\ell}))\}.\]
\end{proof}

So we have established $\rho$-exchange-extremality with respect to $\rho(T)= \eta(T,\rr(T))$, the number of root-containing subtrees. Moreover, if $T$ can be decomposed as $[T_1,\dots,T_k]$, then
\begin{equation}
\eta(T,\rr(T))=\prod_{i=1}^k\left(1+\eta(T_i,\rr(T_i))\right), \label{cha5.sub}
\end{equation}
which is increasing in all of its variables and under addition of further variables, so~\ref{cond1} and~\ref{cond2} hold. Moreover, the minimum of $\eta(T,\rr(T))$, which is trivially equal to 1, is only reached if $T=\bullet$. Thus,~\ref{cond3} holds as well, and Theorem \ref{cha5.thm} implies the following theorem, which was already established in~\cite{Andriantiana1,Zhang1}.

\begin{thm}[\cite{Andriantiana1,Zhang1}]
Given a degree sequence, the corresponding greedy tree is the unique tree that maximises the number of subtrees.
\end{thm}

\subsection{Rooted spanning forests, incidence energy, Laplacian-energy-like invariant}\label{subsec:rsf}

Let $T$ be an $n$-vertex tree and $A(T)$ its adjacency matrix. The \textit{Laplacian matrix} of $T$ is $L(T)=A(T)-D(T)$, where $D(T)$ is the diagonal matrix whose diagonal entry $d_{ii}$ is the degree of the $i$-th vertex $v_i$. The \textit{Laplacian characteristic polynomial} $\Li(T,x)$ is the characteristic polynomial of the Laplacian matrix. In 1967, Kel'mans \cite{kel} gave a combinatorial interpretation for the coefficients of the Laplacian characteristic polynomial as follows. 

\begin{thm}\label{thm:kelmans}
If $\Li(T,x)=\det(xI_n-L(T))=\sum_{k=1}^n (-1)^{n-k}c_k(T)x^k$, then
\begin{align*}
c_k(T)=\sum_{F \in \F(T,k)}\gamma(F),
\end{align*}

where $\F(T,k)$ is the set of all spanning forests of the tree $T$ containing exactly $k$ components and $\gamma(F)$ is the product of the number of vertices in each component of $F$.
\end{thm}

The quantity $\gamma(F)$ can be interpreted as the number of ways to assign roots to the components of a forest $F$, and therefore $c_k(T)$ is the number of $k$-rooted spanning forests of $T$ (spanning forests with $k$ components, where each component is rooted at one of its vertices). 

Let us  consider the polynomial in which we associate to every rooted spanning forest $F$ a weight $x^{\lambda(F)}$, where $\lambda(F)$ is the number of components of $F$. Replacing $x$ by $-x$ in Theorem~\ref{thm:kelmans}, we see that this polynomial is connected to the Laplacian matrix as follows:

\begin{equation}
\rf(T,x)=\sum_{k=1}^n c_k(T)x^k=\det(L(T)+xI_n).\label{cheb}
\end{equation}

Note that the total number of rooted spanning forests in a tree $T$ is equal to $\rf(T,1)$. To avoid confusion as we also consider rooted spanning forests within rooted trees, we will refer to ``marked" spanning forests rather than rooted spanning forests in the following, and call the components' roots ``markers". We show that the invariant $\rf(\cdot,x)$ fits our general scheme, thereby generalising the approach taken in~\cite{Jin}, where trees with given maximum degree are considered.

We define an auxiliary quantity for rooted trees $T$, which is denoted $\f(T,x)$. It counts marked spanning forests $F$ weighted with $x^{\gamma(F)-1}$, in which the root of $T$ is also a marker of one of the forest's components. 
Note that $\f(T,x)$ also counts (weighted) spanning forests of $T$ where all components, except the one containing the root of $T$, have a marker. Finally, we set
\begin{align*}
\rho_1(T,x)=\frac{\rf(T,x)}{\rf(T,x)+\f(T,x)}.
\end{align*}
\begin{lem}
Let $T=[A_1,\dots,A_k]vHw[B_1,\dots,B_{\ell}]$ for some tree $H$, and let $x>0$. Suppose that $\rf(T,x) \leq \rf(T',x)$ for every tree $T'$ with the same degree sequence as $T$. Then 
\begin{itemize}
\item either $k \geq \ell$ and 
\[\min\{\rho_1(A_1,x),\dots,\rho_1(A_k,x)\} \geq \max\{\rho_1(B_1,x),\dots,\rho_1(B_{\ell},x)\},\]

\item or $k \leq \ell$ and
\[\max\{\rho_1(A_1,x),\dots, \rho_1(A_{k},x)\} \leq \min\{\rho_1(B_1,x),\dots,\rho_1(B_{\ell},x)\}.\]
\end{itemize}
\end{lem}

\begin{proof}
Let $A=[A_1,\dots,A_k]$ and $B=[B_1,\dots,B_{\ell}]$. To obtain an expression for $\rf(T,x)$, we consider the following cases for a marked spanning forest of $T$:
\begin{itemize}
\item $\rr(A)$ and $\rr(B)$ belong to components that have a marker in $H$ (possibly the same component). The spanning forests induced in the $A_i$s and $B_i$s are either fully marked or marked except for the root's component (which is joined to the component of $\rr(A)$ or $\rr(B)$ in $H$). Let $c_{11}(H,r)$ be the number of $r$-component marked spanning forests of $H$, and set $a=\sum_{r \geq 1} c_{11}(H,r)x^r$. The contribution of this case to $\rf(T,x)$ is

\begin{equation*}
a\cdot \prod_{i=1}^k(\rf(A_i,x)+\f(A_i,x))\prod_{j=1}^{\ell}(\rf(B_j,x)+\f(B_j,x)).
\end{equation*}

\item $\rr(A)$ and $\rr(B)$ belong to the same component, but this component does not have a marker in $H$. In this case, the marker of the component that contains $\rr(A)$ and $\rr(B)$ lies in one of the $A_i$s or $B_j$s. So we have to choose exactly one of them and replace the factor $\rf(.,x)+\f(.,x)$ by $\rf(.,x)$. Let $c_{00}(H,r)$ be the number of $r$-component spanning forests of $H$, where $\rr(A)$ and $\rr(B)$ lie in the same component and all components except the one containing those are marked. Set $b=\sum_{r \geq 1} c_{00}(H,r)x^{r-1}$. This gives a contribution of

\begin{align*}
b&\cdot \prod_{i=1}^k(\rf(A_i,x)+\f(A_i,x))\prod_{j=1}^{\ell}(\rf(B_j,x)+\f(B_j,x)) \\
&\left(\sum_{i=1}^k\frac{\rf(A_i,x)}{\rf(A_i,x)+\f(A_i,x)}+\sum_{j=1}^{\ell}\frac{\rf(B_j,x)}{\rf(B_j,x)+\f(B_j,x)}\right).
\end{align*} 

\item $\rr(A)$ and $\rr(B)$ lie in different components, both vertices have markers outside of $H$. Now, one of the $A_i$s has to contain the marker of the component of $\rr(A)$, and one of the $B_j$s the marker of the component of $\rr(B)$. Let $c_{00}'(H,r)$ be the number of $r$-component spanning forests of $H$ where $\rr(A)$ and $\rr(B)$ belong to different components and all but those two components are marked. Set $c=\sum_{r \geq 2} c_{00}'(H,r)x^{r-2}$. For this case we get 

\begin{align*}
c&\cdot \prod_{i=1}^k(\rf(A_i,x)+\f(A_i,x))\prod_{j=1}^{\ell}(\rf(B_j,x)+\f(B_j,x)) \\
&\left(\sum_{i=1}^k\frac{\rf(A_i,x)}{\rf(A_i,x)+\f(A_i)}\right)\left(\sum_{j=1}^{\ell}\frac{\rf(B_j,x)}{\rf(B_j,x)+\f(B_j,x)}\right).
\end{align*}

\item $\rr(A)$ and $\rr(B)$ lie in different components, one has a marker in $H$, the other does not. Let $c_{10}(H,r)$ be the number of $r$-component spanning forests of $H$ such that $\rr(A)$ and $\rr(B)$ lie in different components and all but the one containing $\rr(B)$ are marked. Define $c_{01}(H,r)$ analogously, with the roles of $\rr(A)$ and $\rr(B)$ reversed. Now set $d_1=\sum_{r \geq 2} c_{10}(H,r)x^{r-1}$ and $d_2=\sum_{r \geq 2} c_{01}(H,r)x^{r-1}$. Using a similar reasoning as before, we get a contribution of

\begin{align*}
&d_1\cdot\prod_{i=1}^k(\rf(A_i,x)+\f(A_i,x))\prod_{j=1}^{\ell}(\rf(B_j,x)+\f(B_j,x))\cdot\sum_{i=1}^k\frac{\rf(A_i,x)}{\rf(A_i,x)+\f(A_i,x)}  \\&
+d_2\cdot\prod_{i=1}^k(\rf(A_i,x)+\f(A_i,x))\prod_{j=1}^{\ell}(\rf(B_j,x)+\f(B_j,x))\cdot\sum_{j=1}^{\ell}\frac{\rf(B_j,x)}{\rf(B_j,x)+\f(B_j,x)}.
\end{align*}
 
\end{itemize}
Hence, we finally obtain
\begin{align*}
\rf(T,x)=&\prod_{i=1}^k(\rf(A_i,x)+\f(A_i,x))\prod_{j=1}^{\ell}(\rf(B_j,x)+\f(B_j,x))\\
&\Big[a+(b+d_1)\sum_{i=1}^k\frac{\rf(A_i,x)}{\rf(A_i,x)+\f(A_i,x)}+(b+d_2)\sum_{j=1}^{\ell}\frac{\rf(B_j,x)}{\rf(B_j,x)+\f(B_j,x)}\\
&+c\left(\sum_{i=1}^k\frac{\rf(A_i,x)}{\rf(A_i,x)+\f(A_i)}\right)\left(\sum_{j=1}^{\ell}\frac{\rf(B_j,x)}{\rf(B_j,x)+\f(B_j,x)}\right)\Big].
\end{align*}

The product $\prod_{i=1}^k(\rf(A_i,x)+\f(A_i,x))\prod_{j=1}^{\ell}(\rf(B_j,x)+\f(B_j,x))$ remains constant when the $A_i$s and $B_j$s are rearranged, as does the sum

\begin{align*}
\Sigma &=\sum_{i=1}^k\frac{\rf(A_i,x)}{\rf(A_i,x)+\f(A_i,x)}+\sum_{j=1}^{\ell}\frac{\rf(B_j,x)}{\rf(B_j,x)+\f(B_j,x)}\\
&=\sum_{i=1}^k \rho_1(A_i,x)+ \sum_{j=1}^{\ell}\rho_1(B_j,x).
\end{align*}
We can now argue exactly as in the proof of Lemma~\ref{cha5.wie6}. Write $y=\sum_{i=1}^k\rho_1(A_i,x)$. For $\rf(T,x)$ to attain its minimum under all possible permutations of $A_i$s and $B_i$s, the function 
\begin{align*}
y \mapsto a+(b+d_1)y+(b+d_2)(\Sigma-y)+cy(\Sigma-y),
\end{align*}
which is strictly concave in $y$, has to attain its minimum. This occurs when $y$ is either as large or as small as possible. That is,
\begin{itemize}
\item $k \geq \ell$ and $\min\{\rho_1(A_1,x),\dots,\rho_1(A_k,x)\} \geq \max\{\rho_1(B_1,x),\dots,\rho_1(B_{\ell},x)\}$, or 
\item $k \leq \ell$ and $\max\{\rho_1(A_1,x),\dots, \rho_1(A_{k},x)\}\leq \min\{\rho_1(B_1,x),\dots,\rho_1(B_{\ell},x)\}.$
\end{itemize}
\end{proof}

We have established now that minimality with respect to $\rf(\cdot,x)$ implies $\rho_1$-exchange-extremality.
Moreover, the quantity $\rho_1$ can be determined recursively as follows. If $T$ can be decomposed as $[T_1,\dots,T_k]$,  then
\[\rf(T,x)=\prod_{i=1}^k (\rf(T_i,x)+\f(T_i,x))\left(x+\sum_{i=1}^k\frac{\rf(T_i,x)}{\rf(T_i,x)+\f(T_i,x)}\right),\]
and
\[\f(T,x)=\prod_{i=1}^k(\rf(T_i,x)+\f(T_i,x))\]
using similar arguments as before. Thus,
\begin{align}
\rho_1(T,x)=&\frac{x+\sum_{i=1}^k\frac{\rf(T_i,x)}{\rf(T_i,x)+\f(T_i,x)}}{1+x+\sum_{i=1}^k\frac{\rf(T_i,x)}{\rf(T_i,x)+\f(T_i,x)}}
=\frac{x+\sum_{i=1}^k \rho_1(T_i,x)}{1+x+\sum_{i=1}^k \rho_1(T_i,x)}.\label{cha5.eqrow}
\end{align}
The recurrence rule $f_{\rho_1}$ corresponding to \eqref{cha5.eqrow} is increasing in all of its variables and under addition of further variables, moreover $\rho_1(\bullet,x)=\frac{x}{1+x}$ is the unique minimum. So, we may use Theorem \ref{cha5.thm} to obtain the following result.

\begin{thm}\label{cha5.thmro}
For every tree $T$ with degree sequence $D$ and every $x>0$,
\[\rf(T,x)\geq \rf(\G(D),x),\]
with equality if and only if $T$ is isomorphic to $\G(D)$. In particular, $\G(D)$ has the smallest total number of marked spanning forests among trees with degree sequence $D$.
\end{thm}

Now, let us consider other quantities related to the polynomial $\rf(T,x)$.

\begin{defn}
A \textit{subdivision graph}, denoted $S(G)$, is a graph obtained by inserting a new vertex of degree 2 on each edge of $G$.
\end{defn}

\begin{lem}[\cite{Zhou}]\label{lemzhou}
Let $T$ be a tree of order $n$ and $S(T)$ its corresponding subdivision graph. Then
\begin{align*}
c_k(T)=\m(S(T),k), \quad k=0,\dots,n,
\end{align*}
where $\m(S(T),k)$ is the number of $k$-matchings of $S(T)$.
\end{lem}  
Let $\Ma(T,x)=\sum_{k\geq 0}\m(T,k)x^k$ be the matching generating polynomial of $T$; then Lemma~\ref{lemzhou} implies that $\rf(T,x)=\Ma(S(T),x)$. Thus, we obtain the following corollary of Theorem \ref{cha5.thmro}:

\begin{cor}\label{cha5.coro}
For every tree $T$ with degree sequence $D$, and for every $x>0$, \[\Ma(S(T),x)\geq \Ma(S(\G(D)),x),\]
with equality if and only if $T$ is isomorphic to $\G(D)$.
\end{cor}

For a tree $T$, let $\mu_1,\mu_2,\dots,\mu_n$ be the eigenvalues of $\Li(T)$. They are also the eigenvalues of the signless Laplacian matrix $\Li^+(T)$ (this is in fact true for every bipartite graph, see~\cite{cveto}). Recall that the energy of a graph is the sum of the absolute values of the eigenvalues of its adjacency matrix. A variant of the energy, known as the Laplacian-energy-like invariant ($\Lel$ for short, see~\cite{Liu}), is defined by 
\begin{equation*}
\Lel(T)=\sum_{i=1}^n \sqrt{\mu_i}.\label{cha5.inc1}
\end{equation*}
It is closely related to the incidence energy $\IE$ of a graph, defined in \cite{Jooy} as the sum of the singular values of its (vertex-edge) incidence matrix. For every tree $T$, one has $\Lel(T)=\IE(T)$.
Furthermore, it is known that (see \cite{Gutman2})
\begin{equation}
\label{Eq.lap.lemm3}
\Lel(T)=\IE(T)=\frac12 \En(S(T)),
\end{equation}
where $S(T)$ is the subdivision graph of $T$. This allows us to prove the following corollary.

\begin{cor}
Given a degree sequence of a tree, the incidence energy $\IE$, or equivalently the Laplacian-energy-like invariant $\Lel$, is minimised by the greedy tree.
\end{cor}

\begin{proof}
Let $T$ be a tree and $S(T)$ its subdivision graph. Using the Coulson formula \eqref{Eq:EnTree} for the energy and the relation \eqref{Eq.lap.lemm3}, we obtain:

\[\Lel(T)=\IE(T)=\frac{2}{\pi}\int_0^{\infty}\frac{1}{x^2}\ln\left(\sum_{k}\m(S(T),k)x^{2k}\right)\,dx=\frac{2}{\pi}\int_0^{\infty}\frac{1}{x^2}\ln\Ma(S(T),x^2)\,dx.\]
Thus, the claim readily follows from Corollary~\ref{cha5.coro}.
\end{proof}

\subsection{A common generalisation of Wiener index and terminal Wiener index}\label{subsec:Wab} In analogy to the Wiener index, the terminal Wiener index \cite{gfp} is defined as the sum of all distances between pairs of leaves, and the spinal Wiener index \cite{bartlett} is the sum of all distances between pairs of non-leaves. It is known that both are minimised by greedy trees \cite{szekely,bartlett}.
We consider a common generalisation defined as follows: for two fixed positive numbers $a$ and $b$, we set
\[W_{a,b}(T)=\sum_{\{v,w\} \subseteq V(T)}\omega(v)\omega(w)\dis(v,w)=\frac{1}{2}\sum_{v \in V(T)}\sum_{w \in V(T)}\omega(v)\omega(w)\dis(v,w),\]

where
\[\omega(u)=\begin{cases}
a&\text{if $u$ is a leaf,}\\
b&\text{otherwise}.
\end{cases}\]

It is easy to see that $W_{1,1}$ is the Wiener index, $W_{1,0}$ is the terminal Wiener index, and $W_{0,1}$ is the spinal Wiener index. Let us prove an equivalent representation for $W_{a,b}$ that generalises~\eqref{cha5.wie4}.

\begin{prop}\label{Pro3}
We have
\begin{align*}
\sum _{\{v,w\} \subseteq V(T)} \omega(v)\omega(w)\dis(v,w)&=\sum _{vw \in E(T)}\left( \sum _{u\in V(T_v)}\omega(u)\right)\left(\sum _{u \in V(T_w)}\omega(u)\right),
\end{align*}
where $T_v$ and $T_w$ are the components of $T-vw$ containing $v$ and $w$ respectively.
\end{prop}

\begin{proof}
A pair of two vertices $v'$ and $w'$ occurs on the right side of the equation with weight $\omega(v')\omega(w')$ for every edge $vw$ such that $v' \in T_v$ and $w' \in T_w$ (or the other way around). Equivalently, whenever the unique path $P(v', w')$ from $v'$ to $w'$ contains the edge $vw$. Every pair of two vertices $v'$ and $w'$ is counted $\dis(v',w')$ times in this way and therefore contributes $\omega(v')\omega(w') \dis(v',w')$. The sum over all pairs gives the left side of the equation.
\end{proof}

Let $\mathcal{L}(T)$ be the set of leaves of a tree $T$. If $T$ is rooted, the root is only counted as a leaf if it is the only vertex. In view of Proposition~\ref{Pro3}, we can write
\[W_{a,b}(T)=\sum_{vw \in E(T)}\rho_2(T_v)\rho_2(T_w),\]
where $\rho_2(T)=a|\mathcal{L}(T)|+b |V(T)-\mathcal{L}(T)|$. Here, $T_v$ and $T_w$ are regarded as rooted at $v$ and $w$, respectively.

\begin{lem}
Let $T$ be a tree with degree sequence $D$ for which $W_{a,b}(T)$ attains its minimum. For any pair of disjoint complete branches $A=[A_1,\dots,A_k]$ and $B=[B_1,\dots,B_{\ell}]$ in $T$, we have 

\begin{itemize}
\item either $k \geq \ell$ and $\min\{\rho_2(A_1),\dots,\rho_2(A_{k})\} \geq \max\{\rho_2(B_1),\dots,\rho_2(B_{\ell})\}$,

\item or $k \leq \ell$ and $\max\{\rho_2(A_1),\dots, \rho_2(A_{k})\} \leq \min\{\rho_2(B_1),\dots,\rho_2(B_{\ell})\}.$
\end{itemize}\label{cha5.tw1}
\end{lem}

\begin{proof}
The proof is analogous to the one for the Wiener index (Lemma~\ref{cha5.wie6}), using $\rho_2$ instead of $\rho_0$.
\end{proof}

In other words, a tree that minimises $W_{a,b}$ is $\rho_2$-exchange-extremal. Moreover, if $H=[H_1,\dots,H_k]$, then 
\begin{equation}\label{cha5.wie2}
\rho_2(H)=\sum_{i=1}^k\rho_2(H_i) + b.
\end{equation}
We see that conditions \ref{cond1}, \ref{cond2} and \ref{cond3} are satisfied for every fixed pair of positive numbers $a,b$.
Hence, we have the following theorem:

\begin{thm}
Given a degree sequence $D$, $W_{a,b}$ is minimised by $\G(D)$ for every fixed pair of positive numbers $a,b$.
\end{thm}

If we take the limits $a \to 0$ or $b \to 0$, we find that the greedy tree minimises $W_{0,b}$ and $W_{a,0}$, thus in particular the terminal and the spinal Wiener index. However, it may not be the unique optimal tree, since strict inequalities may become non-strict in the limit. Let us, for example, exhibit this phenomenon for the terminal Wiener index $W_{1,0}$.
Consider the trees with degree sequence $(3,2,2,2,2,2,1,1,1)$ shown in Figure~\ref{cha5.tw}.
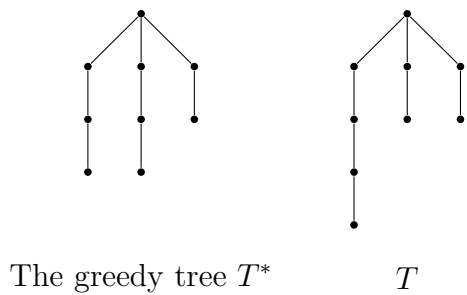
\begin{figure}[H]
\centering
\begin{tikzpicture}[scale=0.7]
\node[fill=black,circle,inner sep=1pt] (t1) at (0,0) {};
\node[fill=black,circle,inner sep=1pt] (t2) at (-1,-1) {};
\node[fill=black,circle,inner sep=1pt] (t3) at (0,-1) {};
\node[fill=black,circle,inner sep=1pt] (t4) at (1,-1) {};
\node[fill=black,circle,inner sep=1pt] (t5) at (-1,-2) {};
\node[fill=black,circle,inner sep=1pt] (t6) at (-1,-3) {};
\node[fill=black,circle,inner sep=1pt] (t7) at (0,-2) {};
\node[fill=black,circle,inner sep=1pt] (t8) at (0,-3) {};
\node[fill=black,circle,inner sep=1pt] (t9) at (1,-2) {};
\draw (t1)--(t2)--(t5)-- (t6);
\draw (t1)--(t3)--(t7)-- (t8);
\draw (t1)--(t4)--(t9);
\node at (0,-5) {The greedy tree $T^*$};
\node[fill=black,circle,inner sep=1pt] (s1) at (5,0) {};
\node[fill=black,circle,inner sep=1pt] (s2) at (4,-1) {};
\node[fill=black,circle,inner sep=1pt] (s3) at (5,-1) {};
\node[fill=black,circle,inner sep=1pt] (s4) at (6,-1) {};
\node[fill=black,circle,inner sep=1pt] (s5) at (4,-2) {};
\node[fill=black,circle,inner sep=1pt] (s6) at (4,-3) {};
\node[fill=black,circle,inner sep=1pt] (s7) at (5,-2) {};
\node[fill=black,circle,inner sep=1pt] (s8) at (4,-4) {};
\node[fill=black,circle,inner sep=1pt] (s9) at (6,-2) {};
\draw (s1)--(s2)--(s5)-- (s6)-- (s8);
\draw (s1)--(s3)--(s7);
\draw (s1)--(s4)--(s9);
\node at (5,-5) {$T$};
\end{tikzpicture}
\caption{Optimal trees for the terminal Wiener index.}
\label{cha5.tw}
\end{figure}
Note that $W_{1,0}(T^*)=6+5+5=16$ and $W_{1,0}(T)=6+6+4=16$, so the greedy tree $T^*$ and the tree $T$ have the same terminal Wiener index and are thus both extremal. Both trees satisfy Lemma~\ref{cha5.tw1}.

\subsection{The Steiner Wiener index}\label{subsec:steiner}
The Steiner distance of a graph, introduced by Chartrand et al.~\cite{chartrand} in 1989, is a natural and nice generalisation of the classical graph distance. For $r \geq 2$, let ${v_1,v_2,\dots,v_r}$ be a set of vertices of a graph $G$. We denote by $S(v_1,v_2\dots,v_r)$ the smallest subtree of $G$ which contains all the vertices $v_1,v_2,\dots,v_r$. Then the \textit{Steiner distance} of $\{v_1,\dots,v_r\}$, denoted by $\sd(v_1,v_2,\dots,v_r)$, is the number of edges in $S(v_1,v_2,\dots,v_r)$.
For $r=2$, the Steiner distance clearly coincides with the classical distance, i.e., $\sd(v_1,v_2)=\dis(v_1,v_2)$.

In \cite{steiner}, the authors define the \textit{Steiner $r$-Wiener index} $\SW_r(G)$ of a graph $G$ as a generalisation of the Wiener index in the following way:
\[\SW_r(G) =\sum_{\{v_1,\dots,v_r\} \subseteq V(G)}\sd(v_1,\dots,v_r).\]
It is straightforward that the case $r=2$ corresponds to the classical Wiener index. We have $\SW_1(G) = 0$ for every graph $G$ and $\SW_n(G) = n-1$ for every $n$-vertex graph $G$, so it is natural to restrict $r$ to the set $\{2,3,\ldots,n-1\}$. There is also an alternative formula for the Steiner Wiener index of trees generalising~\eqref{cha5.wie4}, see \cite{steiner}:

\begin{prop}
Let $T$ be an $n$-vertex tree, and $r$ a positive integer. We have
\begin{align}
SW_r(T)&=\sum_{uv \in E(T)} \sum_{i=1}^{r-1}\binom{|V(T_u)|}{i} \binom{|V(T_v)|}{r-i}\label{sw1}\\
&=\sum_{uv \in E(T)} \Big( \binom{n}{r}-\binom{|V(T_u)|}{r} -\binom{|V(T_v)|}{r} \Big), \label{sw2}
\end{align}
where $T_u$ and $T_v$ are the components of $T-uv$ containing $u$ and $v$ respectively.
\end{prop}

\begin{proof}
Equation \eqref{sw1} is already proven in \cite{steiner}, and Equation \eqref{sw2} follows directly from the Vandermonde identity. There is also a direct combinatorial argument: the edge $uv$ is contained in $S(v_1,v_2,\ldots,v_r)$ if and only if the set $\{v_1,v_2,\ldots,v_r\}$ contains vertices of both $T_u$ and $T_v$. Equivalently, $\{v_1,v_2,\ldots,v_r\}$ can be any set of $r$ vertices that is not a subset of either $V(T_u)$ or $V(T_v)$. Our formula follows immediately.
\end{proof}

Zhang et al. \cite[Question 1.1]{zzwz} asked in a recent paper whether the greedy tree minimises the Steiner $r$-Wiener index for every $r$. In the following, we will answer this question affirmatively. Recall for the following lemma that $\rho_0(T)$ is simply defined to be the number of vertices of $T$. As for the Wiener index, we will show $\rho_0$-exchange-extremality. To avoid degeneracies (see the example below), we consider $\SW_r(T) + \varepsilon W(T)$ for some fixed $\varepsilon > 0$ and later let $\varepsilon$ go to $0$.

\begin{lem}\label{stewie}
Let $\varepsilon$ be a fixed positive real number, and let $T$ be a tree for which $\SW_r(T) + \varepsilon W(T)$ attains its minimum among trees with degree sequence $D$. Then, for any two disjoint complete branches $A=[A_1,\dots,A_k]$ and $B=[B_1,\dots,B_{\ell}]$ in $T$, we have 
\begin{itemize}
\item either $k \geq \ell$ and $\min\{\rho_0(A_1),\dots,\rho_0(A_{k})\} \geq \max\{\rho_0(B_1),\dots,\rho_0(B_{\ell})\}$,

\item or $k \leq \ell$ and $\max\{\rho_0(A_1),\dots, \rho_0(A_{k})\} \leq \min\{\rho_0(B_1),\dots,\rho_0(B_{\ell})\}.$
\end{itemize}
In other words, $T$ is $\rho_0$-exchange-extremal.
\end{lem}

\begin{proof}
We use the same notation as in the proof of Lemma~\ref{cha5.wie6}, see also Figure~\ref{cha5.fig5} again. We have
\begin{align*}
\SW_r(T) + \varepsilon W(T) &= \sum_{j=0}^{t} \Big( \binom{n}{r} - \binom{\alpha+p_j}{r} - \binom{q_{t-j}+\beta}{r} + \varepsilon(\alpha+p_j)(q_{t-j}+\beta) \Big) +C_T, \\
&= \sum_{j=0}^{t} \Big( \binom{n}{r} - \binom{\alpha+p_j}{r} - \binom{n-\alpha-p_j}{r} + \varepsilon(\alpha+p_j)(n-\alpha-p_j) \Big) +C_T,
\end{align*}
where $C_T$ is invariant under permutations of the $A_i$s and $B_i$s. This is a strictly concave function of $\alpha$, so it can only attain its minimum when $\alpha$ is either at its largest or smallest value. As in the proof of Lemma~\ref{cha5.wie6}, this implies the statement.
\end{proof}

We thus have the following theorem.

\begin{thm}
Given a degree sequence $D$, a positive integer $r$ and a positive real number~$\varepsilon$, $\SW(T) + \varepsilon W(T)$ attains its minimum if and only if $T$ is the greedy tree $\G(D)$.
\end{thm}

As mentioned before, we now take the limit $\varepsilon \to 0$ to obtain the following theorem.

\begin{thm}
Given a degree sequence $D$, the greedy tree attains the minimum of the Steiner $r$-Wiener index $\SW_r$.
\end{thm}

As was also the case previously for the terminal Wiener index and the spinal Wiener index, the greedy tree might not be unique with the minimum Steiner $r$-Wiener index. This can happen when $r$ is quite large (compared to the total number of vertices), as in the following example: for the degree sequence $D = (3,2,2,1,1,1)$, both the greedy tree and the only other tree with the same degree sequence have the same Steiner $5$-Wiener index of $27$.

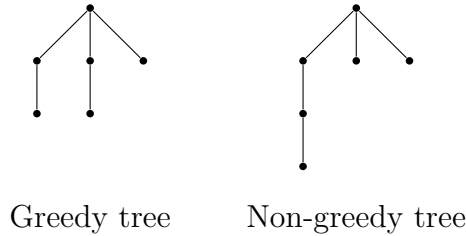
\begin{figure}[H]
\centering
\begin{tikzpicture}[scale=0.7]
\node[fill=black,circle,inner sep=1pt] (t1) at (0,0) {};
\node[fill=black,circle,inner sep=1pt] (t2) at (-1,-1) {};
\node[fill=black,circle,inner sep=1pt] (t3) at (0,-1) {};
\node[fill=black,circle,inner sep=1pt] (t4) at (1,-1) {};
\node[fill=black,circle,inner sep=1pt] (t5) at (-1,-2) {};
\node[fill=black,circle,inner sep=1pt] (t7) at (0,-2) {};
\draw (t1)--(t2)--(t5);
\draw (t1)--(t3)--(t7);
\draw (t1)--(t4);
\node at (0,-4) {Greedy tree};
\node[fill=black,circle,inner sep=1pt] (s1) at (5,0) {};
\node[fill=black,circle,inner sep=1pt] (s2) at (4,-1) {};
\node[fill=black,circle,inner sep=1pt] (s3) at (5,-1) {};
\node[fill=black,circle,inner sep=1pt] (s4) at (6,-1) {};
\node[fill=black,circle,inner sep=1pt] (s5) at (4,-2) {};
\node[fill=black,circle,inner sep=1pt] (s6) at (4,-3) {};
\draw (s1)--(s2)--(s5)-- (s6);
\draw (s1)--(s3);
\draw (s1)--(s4);
\node at (5,-4) {Non-greedy tree};
\end{tikzpicture}
\caption{Both trees with degree sequence $(3,2,2,1,1,1)$ have the same Steiner $5$-Wiener index.}
\end{figure}

\section{Decreasing recurrence rule $f_{\rho}$}\label{recdec}
Let us first mention some definitions, which are needed to describe an $\M$-tree. We use similar terminology and notation as in \cite{Dadah1}. Let us remark that, even though large and small degrees alternate, the concept is slightly different from that of an ``alternating greedy tree'', as introduced in \cite{wang14}.

\begin{defn}
A complete branch $B=[B_1,\dots,B_k]$ of a tree $T$ is called a \textit{pseudo-leaf branch} and its root a \textit{pseudo-leaf} if $|V(B_1)|=|V(B_2)|=\cdots=|V(B_k)|=1$; i.e.,~all vertices in $B$, except for the root $\rr(B)$, are leaves. 
\end{defn}

We simply write $[d]$ for a pseudo-leaf branch with $d$ vertices (a root and $d-1$ leaves). In particular, $[1]$ stands for a single vertex.

\begin{defn}
Let $(d_1,\dots,d_t,1,\dots,1)$ be the degree sequence of a tree $T$, where $d_j \geq 2$ for $1\leq j \leq t$. The $t$-tuple $(d_1,\dots,d_t)$ is called the \textit{reduced degree sequence} of $T$.
We assume that the $d_i$s are in non-increasing order, i.e., $d_1 \geq d_2 \geq \cdots \geq d_t$. 
\end{defn}

By the handshake lemma, the sum of the number of leaves of $T$ and the degrees of the non-leaf vertices, i.e., $|\mathcal{L}(T)|+\sum_{i=1}^td_i$, where $(d_1,\dots,d_t)$ is the reduced degree sequence of $T$, is equal to $2(|\mathcal{L}(T)|+t-1)$. Thus, the number of leaves is 
$
|\mathcal{L}(T)|= 2-2t+\sum_{i=1}^td_i.
$
This implies that two trees with the same reduced degree sequence have the same number of leaves, therefore they have the same degree sequence. 

Next, we give an explicit construction for the $\M$-tree $\M(D)$ with degree sequence $D$. Following the previous remark, we also write $\M(D')$ instead of $\M(D)$ for the same $\M$-tree, where $D'$ is the reduced degree sequence corresponding to $D$. 

\begin{defn}
\label{Def:AltGreed}
Let $(d_1,\dots,d_t)$ be a reduced degree sequence of a tree. If $t\leq d_t+1$, then $\M(d_1,\dots,d_t)$ is the tree obtained by merging $t-1$ leaves of a star $[1+d_t]$ with the roots of the $t-1$ stars $[d_1],[d_2],\dots,[d_{t-1}]$. In our formal notation, it can be written as $[[d_1],[d_2],\ldots,[d_{t-1}],[1],\ldots,[1]]$. We label selected vertices as shown in Figure~\ref{cha5.fig2}, in such a way that
\begin{equation}
d(v_i) \leq d(v_j)\quad \text{if}\, i <j.\label{cha5.eq4}
\end{equation}
At this point all non-leaf vertices are labelled.

\begin{figure}[H]
\centering
\begin{tikzpicture}[scale=1]
\node[fill=black,circle,inner sep=1pt] (v1) at (7,0) {};
\node[fill=black,circle,inner sep=1pt] (v2) at (6.5,-1) {};
\node[fill=black,circle,inner sep=1pt] (v3) at (8.5,-1) {};
\node[fill=black,circle,inner sep=1pt] (v4) at (10,-1) {};
\node[fill=black,circle,inner sep=1pt] (v5) at (11,-1) {};
\node[fill=black,circle,inner sep=1pt] (v6) at (6,-2) {};
\node[fill=black,circle,inner sep=1pt] (v7) at (7,-2) {};
\node[fill=black,circle,inner sep=1pt] (v8) at (8,-2) {};
\node[fill=black,circle,inner sep=1pt] (v9) at (9,-2) {};
\draw (v1)--(v2);
\draw (v1)--(v3);
\draw (v1)--(v4);
\draw (v1)--(v5);
\draw (v2)--(v6);
\draw (v2)--(v7);
\draw (v3)--(v8);
\draw (v3)--(v9);
\node[fill=black,circle,inner sep=0.7pt]  at (7,-1) {};
\node[fill=black,circle,inner sep=0.7pt]  at (7.5,-1) {};
\node[fill=black,circle,inner sep=0.7pt]  at (8,-1) {};
\node[fill=black,circle,inner sep=0.7pt]  at (10.2,-1) {};
\node[fill=black,circle,inner sep=0.7pt]  at (10.4,-1) {};
\node[fill=black,circle,inner sep=0.7pt]  at (10.6,-1) {};
\node[fill=black,circle,inner sep=0.7pt]  at (6.25,-2) {};
\node[fill=black,circle,inner sep=0.7pt]  at (6.5,-2) {};
\node[fill=black,circle,inner sep=0.7pt]  at (6.75,-2) {};
\node[fill=black,circle,inner sep=0.7pt]  at (8.25,-2) {};
\node[fill=black,circle,inner sep=0.7pt]  at (8.5,-2) {};
\node[fill=black,circle,inner sep=0.7pt]  at (8.75,-2) {};
\node at (7,0.3) {$v_1$};
\node at (6.2,-1) {$v_t$};
\node at (8.8,-1) {$v_2$};
\end{tikzpicture}

\caption{Labelling of the vertices of $\M(d_1,\dots,d_t)$ when $t\leq d_t+1$.}
\label{cha5.fig2}
\end{figure}
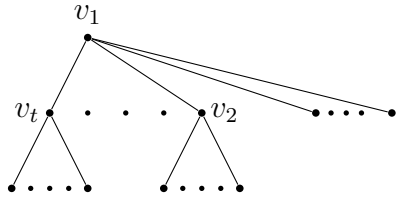

On the other hand, if $t \geq d_t +2$, we construct $\M(d_1,\dots,d_t)$ recursively. Let $\ell$ be the largest integer such that $v_{\ell}$ is a label in $\M(d_{d_t},\dots,d_{t-1})$, and let $s$ be the smallest integer such that $v_s$ is adjacent to a leaf in $\M(d_{d_t},\dots,d_{t-1})$. Let $R_{d_t}=[[d_1],\dots,[d_{d_t-1}]]$, where the pseudo-leaves are labelled $v_{\ell+1},\dots,$ 
$v_{\ell+d_t-1}$, still respecting \eqref{cha5.eq4}. $\M(d_1,\dots,d_t)$ is the tree obtained by merging the root of $R_{d_t}$ with a leaf adjacent to $v_s$.

\begin{figure}[H]
\centering
\begin{tikzpicture}[scale=0.7]
\node[fill=black,circle,inner sep=1pt] (t1) at (0-1,0) {};
\node[fill=black,circle,inner sep=1pt] (t2) at (-2-1,-1) {};
\node[fill=black,circle,inner sep=1pt] (t3) at (0-1,-1) {};
\node[fill=black,circle,inner sep=1pt] (t4) at (2-1,-2) {};
\node[fill=black,circle,inner sep=1pt] (t5) at (-2.5-1,-2) {};
\node[fill=black,circle,inner sep=1pt] (t6) at (-2-1,-2) {};
\node[fill=black,circle,inner sep=1pt] (t7) at (-1.5-1,-2) {};
\node[fill=black,circle,inner sep=1pt] (t8) at (-0.5-1,-2) {};
\node[fill=black,circle,inner sep=1pt] (t9) at (0.5-1,-2) {};
\draw (t1)--(t2);
\draw (t1)--(t3);
\draw (t1)--(t4);
\draw (t2)--(t5);
\draw (t2)--(t6);
\draw (t2)--(t7);
\draw (t3)--(t8);
\draw (t3)--(t9);
\node at (0-1,0.3) {$v_1$};
\node at (-2.3-1,-1) {$v_3$};
\node at (0.3-1,-1) {$v_2$};
\draw[->, line width=1pt] (2-1,-1)--(2.6,-1);
\node[fill=black,circle,inner sep=1pt] (s1) at (5.5,0) {};
\node[fill=black,circle,inner sep=1pt] (s2) at (3.5,-1) {};
\node[fill=black,circle,inner sep=1pt] (s3) at (5.5,-1) {};
\node[fill=black,circle,inner sep=1pt] (s4) at (7.5,-2) {};
\node[fill=black,circle,inner sep=1pt] (s5) at (3,-2) {};
\node[fill=black,circle,inner sep=1pt] (s6) at (3.5,-2) {};
\node[fill=black,circle,inner sep=1pt] (s7) at (4,-2) {};
\node[fill=black,circle,inner sep=1pt] (s8) at (5,-2) {};
\node[fill=black,circle,inner sep=1pt] (s9) at (6,-2) {};
\node[fill=black,circle,inner sep=1pt] (s10) at (7,-3) {};
\node[fill=black,circle,inner sep=1pt] (s11) at (8,-3) {};
\node[fill=black,circle,inner sep=1pt] (s12) at (6.7,-4) {};
\node[fill=black,circle,inner sep=1pt] (s13) at (7,-4) {};
\node[fill=black,circle,inner sep=1pt] (s14) at (7.3,-4) {};
\node[fill=black,circle,inner sep=1pt] (s15) at (7.7,-4) {};
\node[fill=black,circle,inner sep=1pt] (s16) at (8,-4) {};
\node[fill=black,circle,inner sep=1pt] (s17) at (8.3,-4) {};
\draw (s1)--(s2);
\draw (s1)--(s3);
\draw (s1)--(s4);
\draw (s2)--(s5);
\draw (s2)--(s6);
\draw (s2)--(s7);
\draw (s3)--(s8);
\draw (s3)--(s9);
\draw (s4)--(s10);
\draw (s4)--(s11);
\draw (s10)--(s12);
\draw (s10)--(s13);
\draw (s10)--(s14);
\draw (s11)--(s15);
\draw (s11)--(s16);
\draw (s11)--(s17);
\node at (5.5,0.3) {$v_1$};
\node at (3.2,-1) {$v_3$};
\node at (5.8,-1) {$v_2$};
\node at (8.3,-3) {$v_4$};
\node at (6.7,-3) {$v_5$};
\draw[->, line width=1pt] (7.5,-1)--(8+1,-1);
\node[fill=black,circle,inner sep=1pt] (v1) at (11+1,0) {};
\node[fill=black,circle,inner sep=1pt] (v2) at (9+1,-1) {};
\node[fill=black,circle,inner sep=1pt] (v3) at (11+1,-1) {};
\node[fill=black,circle,inner sep=1pt] (v4) at (13.5+1,-2) {};
\node[fill=black,circle,inner sep=1pt] (v5) at (8.5+1,-2) {};
\node[fill=black,circle,inner sep=1pt] (v6) at (9+1,-2) {};
\node[fill=black,circle,inner sep=1pt] (v7) at (9.5+1,-2) {};
\node[fill=black,circle,inner sep=1pt] (v8) at (10.5+1,-2) {};
\node[fill=black,circle,inner sep=1pt] (v9) at (11.5+1,-2) {};
\node[fill=black,circle,inner sep=1pt] (v10) at (13+1,-3) {};
\node[fill=black,circle,inner sep=1pt] (v11) at (14+1,-3) {};
\node[fill=black,circle,inner sep=1pt] (v12) at (12.7+1,-4) {};
\node[fill=black,circle,inner sep=1pt] (v13) at (13+1,-4) {};
\node[fill=black,circle,inner sep=1pt] (v14) at (13.3+1,-4) {};
\node[fill=black,circle,inner sep=1pt] (v15) at (13.7+1,-4) {};
\node[fill=black,circle,inner sep=1pt] (v16) at (14+1,-4) {};
\node[fill=black,circle,inner sep=1pt] (v17) at (14.3+1,-4) {};
\node[fill=black,circle,inner sep=1pt] (v18) at (11.5+1,-3) {};
\node[fill=black,circle,inner sep=1pt] (v19) at (11+1,-4) {};
\node[fill=black,circle,inner sep=1pt] (v20) at (11.3+1,-4) {};
\node[fill=black,circle,inner sep=1pt] (v21) at (11.6+1,-4) {};
\node[fill=black,circle,inner sep=1pt] (v22) at (11.9+1,-4) {};
\draw (v1)--(v2);
\draw (v1)--(v3);
\draw (v1)--(v4);
\draw (v2)--(v5);
\draw (v2)--(v6);
\draw (v2)--(v7);
\draw (v3)--(v8);
\draw (v3)--(v9);
\draw (v4)--(v10);
\draw (v4)--(v11);
\draw (v10)--(v12);
\draw (v10)--(v13);
\draw (v10)--(v14);
\draw (v11)--(v15);
\draw (v11)--(v16);
\draw (v11)--(v17);
\draw (v9)--(v18);
\draw (v18)--(v19);
\draw (v18)--(v20);
\draw (v18)--(v21);
\draw (v18)--(v22);
\node at (11+1,0.3) {$v_1$};
\node at (8.7+1,-1) {$v_3$};
\node at (11.3+1,-1) {$v_2$};
\node at (14.3+1,-3) {$v_4$};
\node at (12.7+1,-3) {$v_5$};
\node at (11.2+1,-3) {$v_6$};
\node at (0-1,-5) {$\M(4,3,3)$};
\node at (5,-5) {$\M(4,4,4,3,3,3)$};
\node at (12+1,-5) {$\M(5,4,4,4,3,3,3,2)$};
\end{tikzpicture}

\caption{Step-by-step construction of $\M(5,4,4,4,3,3,3,2)$.}
\label{cha5.fig3}
\end{figure}
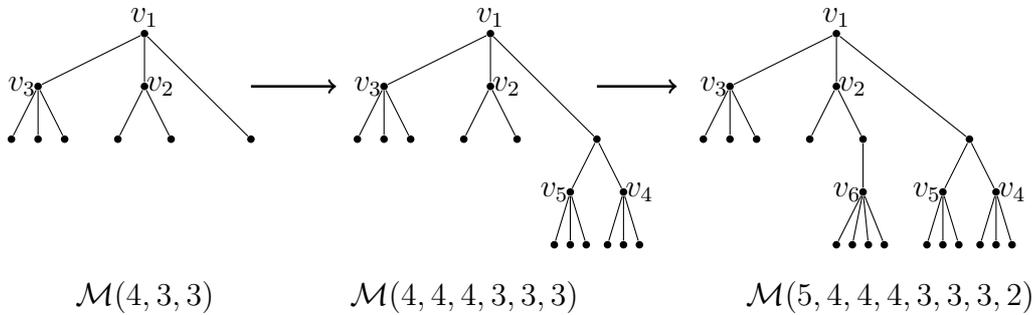
\end{defn}

We will show in this section that the $\M$-tree $\M(D)$ is the unique $\rho$-exchange-extremal tree with degree sequence $D$ if $\rho$ satisfies the following conditions:
\begin{enumerate}[label=\textbf{II.\arabic*}]
\item the quantity $\rho$ satisfies \eqref{cha5.eq1}, where the function $f_{\rho}$ is symmetric,\label{cond11}
\item the function $f_{\rho}$ is strictly decreasing (strictly decreasing in each single variable and strictly decreasing under addition of further variables),\label{cond21}
\item the quantity $\rho$ satisfies
$
\rho(\bullet)>\rho(B),
$
for all rooted trees $B$ with $|V(B)|>1$. \label{cond31}
\end{enumerate}

\subsection{Special case: the Hosoya index}\label{subsec:hosoya}
Let $T$ be a rooted tree, and recall that the Hosoya index, denoted by $z(T)$, is the total number of matchings of $T$. The number of matchings in $T$ that do not cover the root $\rr(T)$ will be denoted by $z_0(T)$. We consider the following ratio, which can be interpreted as the probability that a random matching does not cover the root: 

\begin{equation*}
\rho_3(T)=\frac{z_0(T)}{z(T)}.
\end{equation*}

If $T = [T_1,\dots,T_k]$, then it is not hard to see that
\[z_0(T) = \prod_{j=1}^k z(T_j)\]
and
\[z(T) = z_0(T) + \sum_{j=1}^k z_0(T_j) \prod_{\substack{i=1 \\ i \neq j}}^k z(T_j).\]
It follows that
\begin{equation}\label{eq:recrule3}
\rho_3(T) = \rho_3([T_1,T_2,\dots,T_k])=\frac{1}{1+\sum_{j=1}^k \rho_3(T_j)}.
\end{equation}

Note that this recurrence rule $f_{\rho_3}$ is symmetric and decreasing with respect to each of its variables and under addition of further variables. In addition, we obviously have $\rho_3(\bullet)=1 > \rho_3(B)$ for every rooted tree $B$ with more than one vertex. Thus~\ref{cond11},~\ref{cond21} and~\ref{cond31} are all satisfied.

\begin{lem}[cf.~\cite{Wagn}] 
Let $T$ be a tree with degree sequence $D$ for which $z(T)$ attains its maximum. For any two disjoint complete branches $A=[A_1,\dots,A_k]$ and $B=[B_1,\dots,B_{\ell}]$ in $T$, we have 

\begin{itemize}
\item either $k \geq \ell$ and $\min\{\rho_3(A_1),\dots,\rho_3(A_{k})\} \geq \max\{\rho_3(B_1),\dots,\rho_3(B_{\ell})\}$,

\item or $k \leq \ell$ and $\max\{\rho_3(A_1),\dots, \rho_3(A_k)\} \leq \min\{\rho_3(B_1),\dots,\rho_3(B_{\ell})\}.$
\end{itemize}
\end{lem}

In other words, the tree with degree sequence $D$ that minimises the Hosoya index is $\rho_3$-exchange-extremal. In addition, as a special case of Theorem~22 in \cite{Dadah1}, the following holds.

\begin{thm}[\cite{Dadah1}]
If $T$ is a $\rho_3$-exchange-extremal tree, then $T$ is an $\M$-tree.
\end{thm}

\subsection{Main result}
In analogy to Theorem~\ref{cha5.thm}, we obtain the following result.

\begin{thm}\label{cha5.thmm}
Let $\rho$ be an invariant of rooted trees that satisfies conditions \ref{cond11},\ref{cond21} and \ref{cond31}. If a tree $T$ is $\rho$-exchange-extremal, then $T$ is an $\M$-tree.
\end{thm}

The proof of Theorem~\ref{cha5.thmm} is essentially identical to that of  Theorem~\ref{cha5.thm}, with minor modifications to take into account that $f_{\rho}$ and $f_{\rho_3}$ are decreasing and $\rho(\bullet)$ and $\rho_3(\bullet)$ are maxima.
The rest of this section will be devoted to applications of Theorem~\ref{cha5.thmm}.

\subsection{Matching polynomial and energy}
Let $T$ be a rooted tree. Denote by $\m(T,k)$ the number of $k$-matchings in $T$ and by $\m_0(T,k)$ the number of $k$-matchings not containing the root. We write $\Ma(T,x) = \sum_{k \geq 0} \m(T,k)x^k$ and $\Ma_0(T,x) = \sum_{k \geq 0} \m_0(T,k)x^k$ for the generating polynomials corresponding to $\m$ and $\m_0$ respectively. Moreover, we define the following ratio for $x > 0$:
\begin{equation*}
\tau(T,x)=\frac{\Ma_0(T,x)}{\Ma(T,x)}.
\end{equation*}

\begin{lem}[\cite{Dadah1}]
Let $x>0$, and let $T$ be a tree such that $\Ma(T,x)\leq \Ma(T',x)$ for every tree $T'$ with the same degree sequence. For any two disjoint complete branches $A=[A_1,\dots,A_k]$ and $B=[B_1,\dots,B_{\ell}]$ in $T$, we have 
\begin{itemize}
\item either $k \geq \ell$ and 
$\min\{\tau(A_j,x): 1 \leq j \leq k\} \geq \max\{\tau(B_j,x): 1 \leq j \leq \ell\},$

\item or $k \leq \ell$ and
$\max\{\tau(A_j,x): 1 \leq j \leq k\} \leq \min\{\tau(B_j,x): 1 \leq j \leq \ell\}.$
\end{itemize}
\end{lem}

The tree that minimises the matching polynomial $\Ma(T,x)$ for some $x>0$ is therefore $\tau(.,x)$-exchange-extremal. 
In addition, we have the following generalisation of~\eqref{eq:recrule3}: if $T = [T_1,T_2,\ldots,T_k]$, then

\[
\tau(T,x)=\frac{1}{1+x\sum_{j=1}^k \tau(T_j,x)}.
\]

The recurrence rule $f_{\rho}$, with $\rho(T)=\tau(T,x)$, is thus symmetric and decreasing with respect to each of its variables and under addition of further variables. In addition, $\tau(\bullet,x)=1$ is the unique maximum among all rooted trees for every $x$. Thus the following result from~\cite{Dadah1} can be seen as a special case of Theorem~\ref{cha5.thmm}:

\begin{thm}[\cite{Dadah1}]
Let $x$ be a fixed positive real number. If $T$ is a tree with degree sequence $D$, then
\[\Ma(T,x)\geq \Ma(\M(D),x),\]
and equality holds only if $T$ is isomorphic to $\M(D)$.
\end{thm}

In view of the connection between the energy and the matching polynomial (namely the Coulson formula~\eqref{Eq:EnTree}), we also get the following corollary.

\begin{cor}[\cite{Dadah1}]
Given the degree sequence $D$ of a tree, the energy is minimised by the $\M$-tree $\M(D)$.
\end{cor}

\subsection{Merrifield-Simmons index}\label{subsec:ms}
Recall that the number of independent sets in $T$ is also called the Merrifield-Simmons index and denoted by $\sigma(T)$. Let $\sigma_0(T)$ be the number of independent sets of $T$ that do not contain the root, and define $\rho_4(T)=\sigma(T)/\sigma_0(T)$. The following lemma from \cite{Dadah1} states that the tree that maximises the number of independent sets among all trees with a given degree sequence $D$ has to be $\rho_4$-exchange-extremal. 

\begin{lem}[\cite{Dadah1}]
Let $T$ be a tree with degree sequence $D$ for which $\sigma(T)$ attains its maximum. Then, for any two disjoint complete branches $A=[A_1,\dots,A_k]$ and $B=[B_1,\dots,B_{\ell}]$ in $T$, we have 

\begin{itemize}
\item either
$k \geq \ell$ and 
$\min\{\rho_4(A_j): 1 \leq j \leq k\} \geq \max\{\rho_4(B_j): 1 \leq j \leq \ell\},$

\item or $k \leq \ell$ and
$\max\{\rho_4(A_j): 1 \leq j \leq k\} \leq \min\{\rho_4(B_j): 1 \leq j \leq \ell\}.$
\end{itemize}
\end{lem}

The recurrence rule $f_{\rho_4}$ corresponding to $\rho_4$ is found in a similar way to~\eqref{eq:recrule3}, and is given by
\[
\rho_4([T_1,T_2,\dots,T_k])= 1+\prod_{j=1}^k \rho_4(T_j)^{-1}.
\]
Note that it is decreasing with respect to each of its variables (and under addition of further variables), and $\rho_4(\bullet) = 2$ is easily seen to be maximal among all trees. Hence we can deduce the following theorem as a consequence of Theorem~\ref{cha5.thmm}.

\begin{thm}[\cite{Dadah1}]
Given a degree sequence $D$, the Merrifield-Simmons index is maximised by $\M(D)$.
\end{thm}

\subsection{Solvability}\label{subsec:solv}

Let $G$ be a graph and $v$ a vertex in $G$. The open neighbourhood of $v$, denoted $N(v)$, is the set $\{u \in V(G):uv \in E(G) \}$, and the closed neighbourhood of $v$ is $N[v]=N(v)\cup \{v\}$.

Let $\mathbb{F}_2$ be the field with two elements. The \textit{solvability} of $G$, introduced in \cite{Hatzl} and denoted $s(G)$, is the number of pairs $(a,b) \in \mathbb{F}_2^{V(G)}\times\mathbb{F}_2^{V(G)}$ such that there exists a vector $x\in  \mathbb{F}_2^{V(G)}$  that satisfies
\begin{equation}
(A+\diag(a))x=b,\label{cha5.sol}
\end{equation}
where $A$ is the adjacency matrix of $G$ and $\diag(a)$ is the diagonal matrix whose diagonal is $a$.

This can be interpreted as a domination problem with parity constraints. We are looking for a set $S$ of vertices satisfying, for each of the vertices of the graph, one of four possible conditions: the open/closed neighbourhood has to contain an even/odd number of vertices in $S$. Here, $a$ encodes the open/closed neighbourhood condition and $b$ encodes the required parity. The solvability of a graph measures how many instances of the problem have solutions. It turns out that the solvability of trees can be calculated by a recursion involving a second auxiliary quantity $t$ associated with rooted trees, as stated in the following lemma.

\begin{lem}[\cite{Hatzl}]\label{cha5.sol1}
Let $T$ be a rooted tree. If $T = [T_1,T_2,\ldots,T_k]$, then
\begin{align*}
s(T)&=8\prod_{i=1}^k s(T_i)-5\prod_{i=1}^kt(T_i),\\
t(T)&=8\prod_{i=1}^k s(T_i)-6\prod_{i=1}^kt(T_i),
\end{align*}
with initial values $s(\bullet)=3$ and $t(\bullet)=2$.
\end{lem}

We remark that the value of the parameter $t$ generally depends on the choice of root, unlike $s$.
Note also that $s(T)$ and $t(T)$ are both positive, and that $s(T) > t(T)$ holds for every rooted tree $T$.

It was shown in \cite{Hatzl} that the path has the greatest solvability among all trees (in fact all graphs) of a given order, while the star has the least solvability, but no other extremal results for the solvability of trees are available. Let us show that our general results apply to the solvability as well. To this end, define $\rho_5(T)$ to be the ratio $\rho_5(T)=\frac{s(T)}{t(T)} > 1$. If $T$ can be decomposed as $[T_1,\dots,T_k]$, then by Lemma~\ref{cha5.sol1} we have:

\begin{equation}
\rho_5(T)=\frac{8\prod_{i=1}^k\rho_5(T_i)-5}{8\prod_{i=1}^k\rho_5(T_i)-6}.\label{cha5.sol4}
\end{equation}

Note that this is a symmetric and decreasing recurrence rule, so $\rho_5$ satisfies conditions~\ref{cond11} and~\ref{cond21}. Moreover,~\ref{cond31} is also satisfied by~\eqref{cha5.sol4}, since
\[\rho_5(T)= 1 + \frac{1}{8\prod_{i=1}^k\rho_5(T_i)-6} < 1 + \frac{1}{8-6} = \frac32 = \rho_5(\bullet)\]
for every tree rooted tree $T$ with more than one vertex. Let us now prove the main lemma of this subsection.

\begin{lem}\label{cha5.sol}
Let $T$ be a tree with degree sequence $D$ for which $s(T)$ attains its minimum. Then, for any two disjoint complete branches $A=[A_1,\dots,A_k]$ and $B=[B_1,\dots,B_{\ell}]$ in $T$, we have 
\begin{itemize}
\item either $k \geq \ell$ and $\min\{\rho_5(A_j): 1 \leq j \leq k\} \geq \max\{\rho_5(B_j): 1 \leq j \leq \ell\},$

\item or $k \leq \ell$ and $\max\{\rho_5(A_j): 1 \leq j \leq k\} \leq \min\{\rho_5(B_j): 1 \leq j \leq \ell\}.$
\end{itemize}
\end{lem}

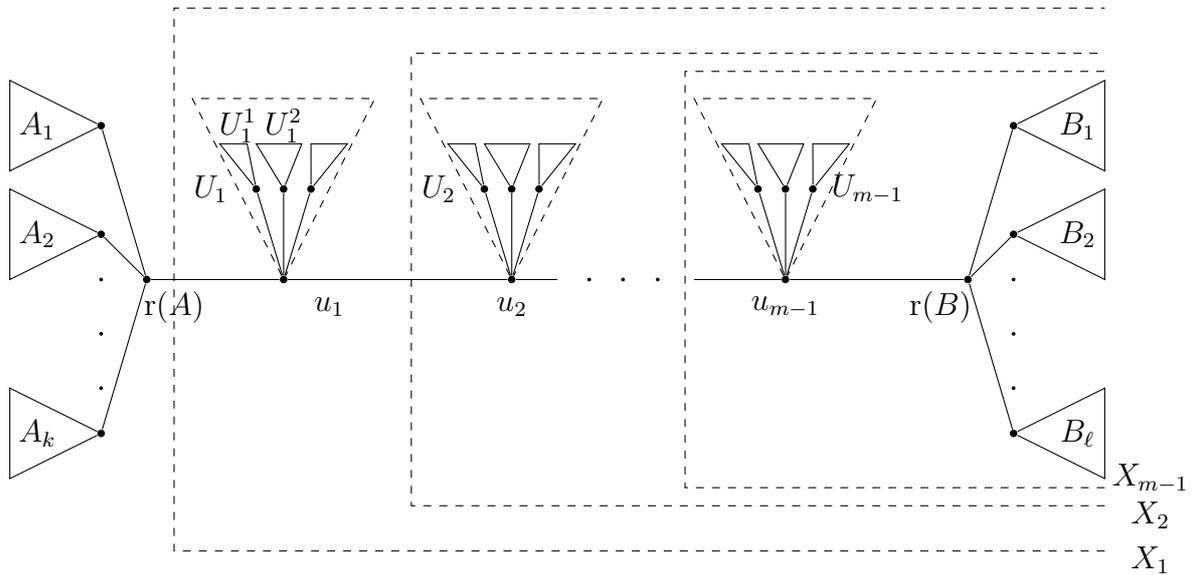
\begin{figure}[H]
\centering
\begin{tikzpicture}[scale=1.2]
\node[fill=black,circle,inner sep=1pt]  (a1) at (0,0.5) {};
\node[fill=black,circle,inner sep=1pt]  (a2) at (0,1.7) {};
\node[fill=black,circle,inner sep=1pt]  (a3) at (0,-1.7) {};
\node[fill=black,circle,inner sep=1pt]  (b1) at (10,0.5) {};
\node[fill=black,circle,inner sep=1pt]  (b2) at (10,1.7) {};
\node[fill=black,circle,inner sep=1pt]  (b3) at (10,-1.7) {};
\node[fill=black,circle,inner sep=1pt]  (u0) at (0.5,0) {};
\node[fill=black,circle,inner sep=1pt]  (um) at (9.5,0) {};
\node[fill=black,circle,inner sep=1pt]  (u1) at (2,0) {};
\node[fill=black,circle,inner sep=1pt]  (u2) at (4.5,0) {};
\node[fill=black,circle,inner sep=1pt]  (um1) at (7.5,0) {};
\node[fill=black,circle,inner sep=1pt]  (v1) at (1.7,1) {};
\node[fill=black,circle,inner sep=1pt]  (v2) at (2,1) {};
\node[fill=black,circle,inner sep=1pt]  (v3) at (2.3,1) {};
\node[fill=black,circle,inner sep=1pt]  (w1) at (4.2,1) {};
\node[fill=black,circle,inner sep=1pt]  (w2) at (4.5,1) {};
\node[fill=black,circle,inner sep=1pt]  (w3) at (4.8,1) {};
\node[fill=black,circle,inner sep=1pt]  (t1) at (7.2,1) {};
\node[fill=black,circle,inner sep=1pt]  (t2) at (7.5,1) {};
\node[fill=black,circle,inner sep=1pt]  (t3) at (7.8,1) {};
\node[fill=black,circle,inner sep=0.5pt]  () at (0,0) {};
\node[fill=black,circle,inner sep=0.5pt]  () at (0,-0.6) {};
\node[fill=black,circle,inner sep=0.5pt]  () at (0,-1.2) {};
\node[fill=black,circle,inner sep=0.5pt]  () at (10,0) {};
\node[fill=black,circle,inner sep=0.5pt]  () at (10,-0.6) {};
\node[fill=black,circle,inner sep=0.5pt]  () at (10,-1.2) {};
\node[fill=black,circle,inner sep=0.5pt]  () at (5.35,0) {};
\node[fill=black,circle,inner sep=0.5pt]  () at (5.72,0) {};
\node[fill=black,circle,inner sep=0.5pt]  () at (6.1,0) {};
\draw (u0)--(u1)--(u2);
\draw (um1)--(um);
\draw (u2)--(5,0);
\draw (um1)--(6.5,0);

\draw[dashed] (u1)--(1,2)--(3,2)-- (u1);
\draw[dashed] (u2)--(3.5,2)--(5.5,2)-- (u2);
\draw[dashed] (um1)--(6.5,2)--(8.5,2)-- (um1);

\draw (u1)--(v1);
\draw (u1)--(v2);
\draw (u1)--(v3);
\draw (v1)--(1.3,1.5)--(1.6,1.5)-- (v1);
\draw (v2)--(1.7,1.5)--(2.2,1.5)-- (v2);
\draw (v3)--(2.3,1.5)--(2.7,1.5)-- (v3);

\draw (u2)--(w1);
\draw (u2)--(w2);
\draw (u2)--(w3);
\draw (w1)--(3.8,1.5)--(4.1,1.5)-- (w1);
\draw (w2)--(4.2,1.5)--(4.7,1.5)-- (w2);
\draw (w3)--(4.8,1.5)--(5.2,1.5)-- (w3);

\draw (um1)--(t1);
\draw (um1)--(t2);
\draw (um1)--(t3);
\draw (t1)--(6.8,1.5)--(7.1,1.5)-- (t1);
\draw (t2)--(7.2,1.5)--(7.7,1.5)-- (t2);
\draw (t3)--(7.8,1.5)--(8.2,1.5)-- (t3);

\draw (a1)--(-1,1)--(-1,0)-- (a1);
\draw (a3)--(-1,-2.2)--(-1,-1.2)-- (a3);
\draw (a2)--(-1,2.2)--(-1,1.2)-- (a2);

\draw (b1)--(11,0)--(11,1)--(b1);
\draw (b3)--(11,-2.2)--(11,-1.2)--(b3);
\draw (b2)--(11,2.2)--(11,1.2)--(b2);

\draw (a1)--(u0)--(a2);
\draw (u0)--(a3);

\draw[dashed] (11,-3)--(0.8,-3)--(0.8,3)--(11,3);
\draw[dashed] (11,-2.5)--(3.4,-2.5)--(3.4,2.5)--(11,2.5);
\draw[dashed] (11,-2.3)--(6.4,-2.3)--(6.4,2.3)--(11,2.3);

\draw (b1)--(um)--(b2);
\draw (um)--(b3);
\node at (2.5,-0.3) {$u_1$};
\node at (4.5,-0.3) {$u_2$};
\node at (7.5,-0.3) {$u_{m-1}$};
\node at (1.2,1) {$U_1$};
\node at (3.7,1) {$U_2$};
\node at (8.4,1) {$U_{m-1}$};
\node at (1.5,1.7) {$U_1^1$};
\node at (2,1.7) {$U_1^2$};
\node at (9.2,-0.3) {$\rr(B)$};
\node at (0.8,-0.3) {$\rr(A)$};
\node at (-0.7,1.7) {$A_1$};
\node at (-0.7,0.5) {$A_2$};
\node at (-0.7,-1.7) {$A_{k}$};
\node at (10.7,1.7) {$B_1$};
\node at (10.7,0.5) {$B_2$};
\node at (10.7,-1.7) {$B_{\ell}$};
\node at (11.5,-3.1) {$X_1$};
\node at (11.5,-2.6) {$X_2$};
\node at (11.5,-2.2) {$X_{m-1}$};
\end{tikzpicture}

\caption{Decomposition of $T$ in the proof of Lemma~\ref{cha5.sol}.}
\label{cha5.figsol}
\end{figure}

\begin{proof}
Let $P=u_0u_1\dots u_m$ be the path between $\rr(A)=u_0$ and $\rr(B)=u_m$. For $1 \leq j<m$, let $U_j=[U_j^1,\dots,U_j^{\rrd(U_j)}]$ be the component containing $u_j$ (rooted at $u_j$) when we remove all the edges of the path $P$. We also denote by $X_j$ the component containing $u_j$ in $T-u_{j-1}u_j$, rooted at $u_j$. We consider the tree $T$ as rooted at $\rr(A)$. Using Lemma~\ref{cha5.sol1}, we can write $s(T)$ and $t(T)$ in terms of matrices as follows:

\begin{align*}
\begin{pmatrix}
s(T)\\t(T)
\end{pmatrix}
=\begin{pmatrix}
8\prod_{i=1}^ks(A_i)&-5\prod_{i=1}^kt(A_i)\\8\prod_{i=1}^ks(A_i)&-6\prod
_{i=1}^kt(A_i)\end{pmatrix}
\begin{pmatrix}
s(X_1)\\t(X_1)
\end{pmatrix}.
\end{align*}

Moreover, for $1 \leq j<m$, we have
\begin{align*}
\begin{pmatrix}
s(X_j)\\t(X_j)
\end{pmatrix}
=\begin{pmatrix}
8\prod_{i=1}^{\rrd(U_j)}s(U_j^i)&-5\prod_{i=1}^{\rrd(U_j)}t(U_j^i)\\
8\prod_{i=1}^{\rrd(U_j)}s(U_j^i)&-6\prod
_{i=1}^{\rrd(U_j)}t(U_j^i)\end{pmatrix}\begin{pmatrix}
s(X_{j+1})\\t(X_{j+1})
\end{pmatrix}.
\end{align*}

Let us denote the $2 \times 2$-matrix on the right side of this equation by $M_j$. Empty products (if $\rrd(U_j) = 0$) are defined to be $1$, which is consistent with the recursion. Then we have

\begin{align*}
\begin{pmatrix}
s(T)\\t(T)
\end{pmatrix}
=\begin{pmatrix}
8\prod_{i=1}^ks(A_i)&-5\prod_{i=1}^kt(A_i)\\8\prod_{i=1}^ks(A_i)&-6\prod
_{i=1}^kt(A_i)\end{pmatrix}
M
\begin{pmatrix}
s(B)\\t(B)
\end{pmatrix},
\end{align*}
where $M=M_0\times M_1\times M_2 \times \cdots \times M_{m-1}$, $M_0$ being the identity matrix. If we set $M=\begin{pmatrix}
M_{00}&M_{01}\\M_{10}&M_{11}
\end{pmatrix}$, then we can write the solvability of $T$ as

\begin{align*}
s(T)&=\prod_{i=1}^ks(A_i)\prod_{i=1}^{\ell}s(B_i)(64M_{00}+64M_{01}) - \prod_{i=1}^ks(A_i)\prod_{i=1}^{\ell}t(B_i)(40M_{00}+48M_{01})\\
&\quad-\prod_{i=1}^kt(A_i)\prod_{i=1}^{\ell}s(B_i)(40M_{10}+40M_{11})+\prod_{i=1}^kt(A_i)\prod_{i=1}^{\ell}t(B_i)(25M_{10}+30M_{11})\\
&=\prod_{i=1}^kt(A_i)\prod_{i=1}^{\ell}t(B_i)\Big[(64M_{00}+64M_{01})\prod_{i=1}^k\rho_5(A_i)\prod_{i=1}^{\ell}\rho_5(B_i)-
(40M_{00}+48M_{01})\prod_{i=1}^{k}\rho_5(A_i)\\
&\quad-(40M_{10}+40M_{11})\prod_{i=1}^{\ell}\rho_5(B_i)+(25M_{10}+30M_{11})\Big].
\end{align*}

Since $\prod_{i=1}^kt(A_i)\prod_{i=1}^{\ell}t(B_i)$ and $(64M_{00}+64M_{01})\prod_{i=1}^k\rho_5(A_i)\prod_{i=1}^{\ell}\rho_5(B_i) + (25M_{10}+30M_{11})$ remain invariant under any rearrangements of the $A_i$s and $B_i$s, $s(T)$ attains its minimum under permutations of the $A_i$s and $B_i$s if and only if
\[\alpha\prod_{i=1}^{k}\rho_5(A_i)+\beta\prod_{i=1}^{\ell}\rho_5(B_i)\]
attains its maximum, where $\alpha=40M_{00}+48M_{01}$ and $\beta=40M_{10}+40M_{11}$.
 
Next we show that $\alpha$ and $\beta$ are positive. For $m=1$, $M$ is equal to the identity matrix. So $\alpha=\beta=40>0$.

Now, for $m>1$, let us prove by induction on $m$ that $5M_{00}+6M_{01} > 0$, $5M_{10}+6M_{11}>0$ and $M_{01},M_{11}\leq 0$. Note that the positivity of $5M_{00}+6M_{01}$ implies the positivity of $\alpha$ since $\alpha=8(5M_{00}+6M_{01})$, and the positivity of $5M_{10}+6M_{11}$, combined with the inequality $M_{11} \leq 0$, implies the positivity of $\beta$ since $\beta=8(5M_{10}+5M_{11})\geq 8(5M_{10}+6M_{11})$. Let us write 
\[\begin{bmatrix} M_{00}(m) & M_{01}(m) \\ M_{10}(m) & M_{11}(m) \end{bmatrix} =  
M(m) = M_0 \times M_1 \times \cdots \times M_{m-1}.\]
For $m=2$, we have, since $s(U_1^i)>t(U_1^i)> 0$ for all $i$, 
\begin{align*}
M_{01}(2)&=-5\prod_{i=1}^{\rrd(U_1)}t(U_1^i)\leq 0,\quad
M_{11}(2)=-6\prod_{i=1}^{\rrd(U_1)}t(U_1^i)\leq 0,\\
&5M_{00}(2)+6M_{01}(2)=40\prod_{i=1}^{\rrd(U_1)}s(U_1^i)
-30\prod_{i=1}^{\rrd(U_1)}t(U_1^i)>0,\\
&5M_{10}(2)+6M_{11}(2)=40\prod_{i=1}^{\rrd(U_1)}s(U_1^i)
-36\prod_{i=1}^{\rrd(U_1)}t(U_1^i)>0.
\end{align*}
Suppose now that the statement is true for $m =d$. For the induction step, we take $m=d+1$, and $M(d+1)=M_0\times M_1 \times \cdots \times M_{d}=M(d) \times M_{d}$. We have

\begin{align*}
M_{00}(d+1)&=  (8M_{00}(d)+8M_{01}(d))\prod_{i=1}^{\rrd(U_{d+1})}s(U_{d+1}^i),\\
M_{01}(d+1)&=-(5M_{00}(d)+6M_{01}(d))\prod_{i=1}^{\rrd(U_{d+1})}t(U_{d+1}^i),\\
M_{10}(d+1)&=(8M_{10}(d)+8M_{11}(d))\prod_{i=1}^{\rrd(U_{d+1})}s(U_{d+1}^i),\\
M_{11}(d+1)&=-(5M_{10}(d)+6M_{11}(d))\prod_{i=1}^{\rrd(U_{d+1})}t(U_{d+1}^i).
\end{align*}

By the induction hypothesis, we have $M_{01}(d),M_{11}(d)\leq 0$, $5M_{00}(d)+6M_{01}(d) > 0$ and $5M_{10}(d)+6M_{11}(d)>0$. For all $i$, $s(U_{d+1}^i) > t(U_{d+1}^i) > 0$, so
\[M_{01}(d+1) \leq 0, \text{ and } M_{11}(d+1) \leq 0.\]
Moreover, we get

\begin{align*}
&5M_{00}(d+1)+6M_{01}(d+1)\\
&=8 \cdot (5M_{00}(d)+5M_{01}(d))\prod_{i=1}^{\rrd(U_{d+1})}s(U_{d+1}^i)-6 \cdot (5M_{00}(d)+6M_{01}(d))\prod_{i=1}^{\rrd(U_{d+1})}t(U_{d+1}^i)\\
&\geq \left(5M_{00}(d)+6M_{01}(d)\right)\left(8\prod_{i=1}^{\rrd(U_{d+1})}s(U_{d+1}^i)-6\prod_{i=1}^{\rrd(U_{d+1})}t(U_{d+1}^i)\right)>0
\end{align*}
and
\begin{align*}
&5M_{10}(d+1)+6M_{11}(d+1)\\
&=8 \cdot (5M_{10}(d)+5M_{11}(d))\prod_{i=1}^{\rrd(U_{d+1})}s(U_{d+1}^i)-6 \cdot (5M_{10}(d)+6M_{11}(d))\prod_{i=1}^{\rrd(U_{d+1})}t(U_{d+1}^i)\\
&\geq \left(5M_{10}(d)+6M_{11}(d)\right)\left(8\prod_{i=1}^{\rrd(U_{d+1})}s(U_{d+1}^i)-6\prod_{i=1}^{\rrd(U_{d+1})}t(U_{d+1}^i)\right) >0.
\end{align*}
This completes our induction. Since we know now that $\alpha$ and $\beta$ are positive, and that $\rho_5(A_1),\dots,\rho_5(A_k),\rho_5(B_1),\dots,\rho_5(B_{\ell})$ are all greater than $1$ by definition, Lemma~\ref{Lem:RearrIneq} applies and shows that $s(T)$ can only be minimal if
\begin{itemize}
\item $\alpha \geq \beta$, $k \geq \ell$ and 
$\min\{\rho_5(A_j): 1 \leq j \leq k\} \geq \max\{\rho_5(B_j): 1 \leq j \leq \ell\}$,
\item or $\alpha \leq \beta$, $k \leq \ell$ and
$\max\{\rho_5(A_j): 1 \leq j \leq k\} \leq \min\{\rho_5(B_j): 1 \leq j \leq \ell\}$.
\end{itemize}
This completes the proof.
\end{proof}

So a tree that minimises the solvability, given the degree sequence, must be $\rho_5$-exchange-extremal. Applying Theorem~\ref{cha5.thmm}, we immediately get the following theorem.

\begin{thm}
Given a degree sequence of a tree $D$, the solvability is minimised by $\M(D)$.
\end{thm}

\section{Majorisation of degree sequences}\label{bounded}
We say that the sequence $A=(a_1,a_2,\dots,a_n)$ majorises $B=(b_1,b_2,\dots,b_n)$ if $a_1 + a_2 + \cdots + a_n = b_1 + b_2 + \cdots + b_n$ and $a_1+a_2+\cdots+a_k\geq b_1+b_2+\cdots+b_k$ for all $k$ with $1\leq k < n$. In this case we write $A\triangleright B$. Let $\mathbb{T}_n$ be the set of $n$-vertex trees. In this section we study the class of trees whose degree sequence is majorised by a fixed degree sequence $B$ of a tree, i.e.,
\[
\mathbb{T}_{\triangleleft B}=\{T\in \mathbb{T}_n: B \triangleright D, \: \text{where $D$ is the degree sequence of $T$}\}.
\]
Let $\rho$ be an invariant of rooted trees which satisfies a recursive relation as in Equation \eqref{cha5.eq1}, and let $\mathbb{S}_n$ denote the set of permutations of $1,2,\dots,n$. We say that the tree invariant $I$ is maximum-$\rho$-compatible if the following holds:

For every tree $H$ and every choice of two of its leaves $v$ and $w$, every collection of rooted trees $T_1,T_2,\dots ,T_r$ with $\rho(T_1)\geq \rho(T_2)\geq \dots \geq \rho(T_r)$, and every integer $s$ with $r/2 \leq s \leq r$, the maximum value of $I$ among all trees in the set 
\begin{align*}
\mathbb{K}_s=\{[T_{\sigma(1)},\dots,T_{\sigma(k)}]vHw[T_{\sigma(k+1)},\dots,T_{\sigma(r)}]: \sigma\in\mathbb{S}_r\text{ and }r-s\leq k\leq s\}
\end{align*}
is attained by (at least) one of the following two trees:
\[[T_1,\dots,T_s]vHw[T_{s+1},\dots,T_r]\text{ or }[T_{s+1},\dots,T_r]vHw[T_1,\dots,T_s].\]
In words, among all possibilities to attach $T_1,T_2,\dots,T_r$ to $v$ and $w$ in such a way that their degrees are not greater than some fixed bound $s+1$, the maximum value of $I$ is reached when the $s$ trees $T_i$ with largest $\rho$-values are attached to one of the two vertices~$v$ and~$w$, so that its degree is $s+1$. Note here that if $s\leq s'$ then $\mathbb{K}_s\subseteq \mathbb{K}_{s'}$. In the same way, one defines an invariant $I$ to be minimum-$\rho$-compatible (replacing ``maximum value'' by ``minimum value'' in the definition). 

Several examples from the previous sections fit this scheme. In each case, the proof is essentially the same as the proof that extremality with respect to $I$ implies $\rho$-exchange-extremality.

\begin{itemize}
\item The Wiener index is minimum-$\rho_0$-compatible, and this extends to the generalisations discussed in Subsections~\ref{subsec:Wab} and~\ref{subsec:steiner}. 
\item The number of subtrees is maximum-$\rho$-compatible with respect to the invariant $\rho(T)= \eta(T,\rr(T))$ defined in Subsection~\ref{subsec:subtrees}.
\item The weighted number of rooted spanning forests $\rf(T,x)$ discussed in Subsection~\ref{subsec:rsf} is minimum-$\rho_1$-compatible with respect to the invariant $\rho_1$ defined there.
\item The Hosoya index is minimum-$\rho_3$-compatible (see Subsection~\ref{subsec:hosoya}), and this extends to the matching generating polynomial $\Ma(T,x)$.
\item The Merrifield-Simmons index is maximum-$\rho_4$-compatible (Subsection~\ref{subsec:ms}).
\item The solvability is minimum-$\rho_5$-compatible (Subsection~\ref{subsec:solv}).
\end{itemize}

We note that for any maximum-$\rho$-compatible invariant $I$, a tree with degree sequence $D$ that attains the maximum value of $I$ is necessarily $\rho$-exchange-extremal, since the compatibility condition is an extension of $\rho$-exchange-extremality.

\begin{thm}
\label{Th:BoundedDegSeq}
Let $I$ be a maximum-$\rho$-compatible tree invariant, for some $\rho$ that satisfies \eqref{cha5.eq1}, and let $B$ be a degree sequence of a tree. Then 
\begin{itemize}
\item[i)] $\max\{I(T):T\in\mathbb{T}_{\triangleleft B}\}=I(\G(B))$ if $\rho$ satisfies properties~\ref{cond1},~\ref{cond2} and~\ref{cond3}, 
\item[ii)] $\max\{I(T):T\in\mathbb{T}_{\triangleleft B}\}=I(\M(B))$ if $\rho$ satisfies properties~\ref{cond11},~\ref{cond21} and~\ref{cond31}, 
\end{itemize}
\end{thm}

\begin{proof}
We only prove part i), the proof of part ii) is analogous. As mentioned earlier, by definition of maximum-$\rho$-compatibility of $I$, $\max\{I(T):T\in\mathbb{T}_{\triangleleft B}\}$ is reached by a $\rho$-exchange-extremal tree, say $E$. It is only left to prove that $E$ can be chosen to have degree sequence $B$. Let $T\in\mathbb{T}_{\triangleleft B}$, with a degree sequence $D$ (thus $D\triangleleft B$), such that $D=(d_1,d_2,\dots,d_n) \neq B=(b_1,b_2,\dots,b_n)$. Let $r$ be the smallest index for which $d_r > b_r$ (such an index must exist, since $B$ and $D$ have the same sum), and let $\ell$ be the largest index less than $r$ such that $d_{\ell} < b_{\ell}$ (such an index must exist, since $b_1 + b_2 + \cdots b_r \geq d_1 + d_2 + \cdots + d_r$ by majorisation). Observe that $d_i = b_i$ for $\ell < i < r$. Now define $B_1 = (b'_1,b'_2,\dots,b'_n)$ by $b'_i = b_i$ for $i \notin \{\ell,r\}$, $b'_{\ell} = b_{\ell}-1$ and $b'_r = b_r+1$. We note that $B_1$ is still non-increasing, thus a valid degree sequence: $b'_{\ell} = b_{\ell}-1 \geq d_{\ell} \geq d_{\ell+1} \geq b'_{\ell+1}$ and $b'_r = b_r+1 \leq d_r \leq d_{r-1} \leq b'_{r-1}$. Moreover, it is easy to verify that $D \triangleleft B_1$, since $b'_i \geq d_i$ for $i < r$, while $b'_1 + b'_2 + \cdots + b'_j = b_1 + b_2 + \cdots + b_j$ for $j \geq r$. Clearly, we also have $B_1 \triangleleft B$. We can repeat this procedure with $B_1$ and $D$ to obtain a degree sequence $B_2$ with $D \triangleleft B_2 \triangleleft B_1$, and so on. This yields a sequence of degree sequences $B = B_0,B_1,B_2,\ldots,B_p = D$ such that
\[B = B_0 \triangleright B_1 \triangleright \cdots \triangleright B_p = D,\]
and $B_i$ and $B_{i+1}$ always only differ in two entries (by exactly $1$ each). We therefore restrict ourselves first to the case that $B$ and $D$ only differ in two entries: $b_{\ell} = d_{\ell}+1$, $b_r = d_r -1$ and $b_i = d_i$ for $i \notin \{\ell,r\}$.

Let $v$ and $w$ be distinct vertices in $T$ such that $d(v)=d_{\ell} \geq d_r= d(w)$.
Then $T$ can be decomposed as $[T_1,\dots,T_{d_{\ell}-1}]vHw[T_{d_{\ell}},\dots,T_{d_{\ell}+d_r-2}]$ for some tree $H$ and branches $T_1,T_2,\dots$ $,T_{d_{\ell}+d_r-2}$. Let $\sigma\in \mathbb{S}_{d_{\ell}+d_r-2}$ such that $\rho(T_{\sigma(1)})\geq \rho(T_{\sigma(2)})\geq \dots\geq \rho(T_{\sigma(d_{\ell}+d_r-2)})$. 

By the definition of maximum-$\rho$-compatibility, we can find a tree $T'$, which is either
\begin{align*}
&[T_{\sigma(1)},\dots,T_{\sigma(d_{\ell})}]vHw[T_{\sigma(d_{\ell}+1)},\dots,T_{\sigma(d_{\ell}+d_r-2)}]\quad \text{or}\\
&[T_{\sigma(d_{\ell}+1)},\dots,T_{\sigma(d_{\ell}+d_r-2)}]vHw[T_{\sigma(1)},\dots,T_{\sigma(d_{\ell})}],
\end{align*}
such that its degree sequence is $B$ and $I(T) \leq I(T')$. In the general case, we can find a sequence of trees $T_0,T_1,\ldots,T_p = T$ whose degree sequences are $B = B_0, B_1,\ldots,B_p = D$ respectively, such that
\[I(T_0) \geq I(T_1) \geq \cdots \geq I(T_p) = I(T).\]
This shows that the extremal tree $E$ mentioned earlier can indeed be chosen to have degree sequence $B$. The statement of the theorem follows.
\end{proof}

\begin{cor}\label{corbounded}
Let $\mathbb{A}\subseteq \mathbb{T}_n$ be a class of trees with $n$ vertices and $B$ a degree sequence such that for every tree $T\in \mathbb{A}$, its degree sequence $D$ satisfies $D \triangleleft B$. If $I$ is a maximum-$\rho$-compatible tree invariant, then
\begin{itemize}
\item $\max\{I(T):T\in \mathbb{A}\}=I(\G(B))$
if $\rho$ satisfies~\ref{cond1},~\ref{cond2} and~\ref{cond3} and $\G(B)\in \mathbb{A}$,
\item $\max\{I(T):T\in \mathbb{A}\}=I(\M(B))$
if $\rho$ satisfies~\ref{cond11},~\ref{cond21} and~\ref{cond31} and $\M(B)\in \mathbb{A}.$
\end{itemize}
Again, an analogous statement holds if $I$ is minimum-$\rho$-compatible (replacing $\max$ by $\min$).
\end{cor}

\begin{proof}
The degree sequence condition implies that $\mathbb{A} \subseteq \mathbb{T}_{\triangleleft B}$, so the statement is immediate from Theorem~\ref{Th:BoundedDegSeq}.
\end{proof}

We conclude this paper by discussing different applications of Corollary~\ref{corbounded}.

\begin{itemize}
\item The first example is the class of all trees of order $n$, $\mathbb{A}=\mathbb{T}_n$, where we can take $B=(n-1,1,\dots,1)$. The only tree with this degree sequence is the star $S_n=\G(B)=\M(B)$. Therefore, we obtain the following known results as immediate corollaries: among all $n$-vertex trees, the star minimises the Wiener index \cite{dobrynin}, the Hosoya index and the energy \cite{gutman}, the quantity $\rf(T,x)$ \cite{Zhou} for every $x>0$ (thus also the incidence energy) and the solvability \cite{Hatzl}, while it maximises the number of subtrees \cite{Szek2} and the Merrifield-Simmons index \cite{Prodinger101}. 

\item Next, we take $\mathbb{A}$ to be the set of all $n$-vertex trees whose vertex degrees are at most equal to $d$. Here, we can take $B=(d,\dots,d,r,1,\dots,1)$, where $1\leq r< d$ and $r \equiv n-1 \mod d-1$. The greedy tree $\G(B)$ is in this case called a Volkmann tree. Again, several extremality results now follow as direct applications of Corollary~\ref{corbounded}. For example, the Volkmann tree minimises the Wiener index \cite{fischermann} and the polynomial $\rf(T,x)$ for every $x > 0$ \cite{Jin} (thus also the incidence energy), while it maximises the number of subtrees \cite{kirk}. On the other hand, the $\M$-tree $\M(B)$ is known to be extremal for the Hosoya index, the energy of a graph and the Merrifield-Simmons index \cite{Heub48,He}. We also obtain new results for the Steiner Wiener index (which is minimised by $\G(B)$) and the solvability (which is minimised by $\M(B)$).

\item Next, let $\mathbb{A}$ be the set of $n$-vertex trees with exactly $\ell$ leaves. Now we can take $B=(\ell,2,\ldots,2,1,\ldots,1)$ ($n-\ell-1$ copies of $2$, $\ell$ copies of $1$) and find that either $\G(B)$ or $\M(B)$ is extremal for each of the invariants we considered in the previous sections, thus recovering several known results (and adding some new ones). For instance, the greedy tree $\G(B)$ is known to minimise the Wiener index \cite{dobrynin,wang10} and the polynomial $\rf(T,x)$ as well as the incidence energy \cite{ilic}. On the other hand, it maximises the number of subtrees \cite{Andriantiana1}. Similarly, the extremal tree for the Hosoya index, the Merrifield-Simmons index and the energy is known to be $\M(B)$ \cite{Yu,yeyuan,yulv}. We find now also that $\G(B)$ minimises the Steiner Wiener index, while $\M(B)$ minimises the solvability.

\item Finally, let $\mathbb{A}$ be the set of $n$-vertex trees having $r$ branching vertices, i.e., vertices of degree greater than or equal to $3$. Here, we can take $B = (n-2r+1,3,\ldots,3,1,\ldots,1)$ ($r-1$ copies of $3$, $n-r$ copies of $1$). It was shown in~\cite{Hong} that the greedy tree $\G(B)$ minimises the Wiener index in this class of trees if $r \leq \frac{n+2}{3}$, while the case $r > \frac{n+2}{3}$ was left as an open problem. Minimality of $\G(B)$ with respect to the Wiener index is a consequence of our results for arbitrary $r$ now, and we also obtain extremality of $\G(B)$ or $\M(B)$ for all the other invariants mentioned in this paper.
 
\end{itemize}

\bibliographystyle{abbrv} 
\bibliography{Extremal}

\begin{thebibliography}{10}

\bibitem{Dadah1}
E.~O.~D. Andriantiana.
\newblock Energy, {H}osoya index and {M}errifield-{S}immons index of trees with
  prescribed degree sequence.
\newblock {\em Discrete Appl. Math.}, 161:724--741, 2013.

\bibitem{Andriantiana2}
E.~O.~D. Andriantiana and S.~Wagner.
\newblock Spectral moments of trees with given degree sequence.
\newblock {\em Linear Algebra Appl.}, 439(12):3980--4002, 2013.

\bibitem{Andriantiana1}
E.~O.~D. Andriantiana, S.~Wagner, and H.~Wang.
\newblock Greedy trees, subtrees and antichains.
\newblock {\em Electron. J. Combin.}, 20(3):P28, 2013.

\bibitem{bartlett}
M.~Bartlett, E.~Krop, C.~Magnant, F.~Mutiso, and H.~Wang.
\newblock Variations of distance-based invariants of trees.
\newblock {\em J. Combin. Math. Combin. Comput.}, 91:19--29, 2014.

\bibitem{biyi}
T.~B{\i}y{\i}koglu and J.~Leydold.
\newblock Graphs with given degree sequence and maximal spectral radius.
\newblock {\em Electron. J. Combin.}, 15:R119, 2008.

\bibitem{chartrand}
G.~Chartrand, O.~R. Oellermann, S.~Tian, and H.~Zou.
\newblock {S}teiner distance in graphs.
\newblock {\em \~{C}asopis Pest. Mat.}, 114:399--410, 1989.

\bibitem{cgjz}
Y.-H. Chen, D.~Gray, Y.-L. Jin, and X.-D. Zhang.
\newblock On majorization of closed walk vectors of trees with given degree
  sequences.
\newblock {\em Appl. Math. Comput.}, 336:326--337, 2018.

\bibitem{cveto}
D.~Cvetkovi{\'c}, P.~Rowlinson, and S.~Simi{\'c}.
\newblock Signless {L}aplacians of finite graphs.
\newblock {\em Lin. Algebra Appl.}, 423:155--171, 2007.

\bibitem{dobrynin}
A.~A. Dobrynin, R.~Entringer, and I.~Gutman.
\newblock Wiener index of trees: theory and applications.
\newblock {\em Acta Appl. Math.}, 66(3):211--249, 2001.

\bibitem{fischermann}
M.~Fischermann, A.~Hoffmann, D.~Rautenbach, L.~Sz{\'e}kely, and L.~Volkmann.
\newblock Wiener index versus maximum degree in trees.
\newblock {\em Discrete Appl. Math.}, 122(1):127--137, 2002.

\bibitem{gutman}
I.~Gutman.
\newblock Acyclic systems with extremal {H}{\"u}ckel $\pi$-electron energy.
\newblock {\em Theor. Chim. Acta}, 45:79--87, 1977.

\bibitem{gfp}
I.~Gutman, B.~Furtula, and M.~Petrovi\'{c}.
\newblock Terminal {W}iener index.
\newblock {\em J. Math. Chem.}, 46(2):522--531, 2009.

\bibitem{Gutman2}
I.~Gutman, D.~Kiani, and M.~Mirzakhah.
\newblock On incidence energy of graphs.
\newblock {\em MATCH Commun. Math. Comput. Chem.}, 62:573--580, 2009.

\bibitem{Gutman3}
I.~Gutman and O.~E. Polansky.
\newblock {\em Mathematical {C}oncepts in {O}rganic {C}hemistry}.
\newblock Springer, Berlin, 1986.

\bibitem{Hatzl}
J.~Hatzl and S.~Wagner.
\newblock Combinatorial properites of a general domination problem with parity
  constraints.
\newblock {\em Discrete Math.}, 308:6355--6367, 2008.

\bibitem{He}
C.~Heuberger and S.~Wagner.
\newblock Maximizing the number of independent subsets over trees with bounded
  degree.
\newblock {\em J. Graph Theory}, 58(1):49--68, 2008.

\bibitem{Wagn}
C.~Heuberger and S.~Wagner.
\newblock Chemical trees minimizing energy and {H}osoya index.
\newblock {\em J. Math. Chem.}, 46:214--230, 2009.

\bibitem{Heub48}
C.~Heuberger and S.~Wagner.
\newblock On a class of extremal trees for various indices.
\newblock {\em MATCH Commun. Math. Comput. Chem.}, 62:437--464, 2009.

\bibitem{ilic}
A.~Ili\'{c} and M.~Ili\'{c}.
\newblock Laplacian coefficients of trees with given number of leaves or
  vertices of degree two.
\newblock {\em Linear Algebra Appl.}, 431(11):2195--2202, 2009.

\bibitem{Jin}
Y.-L. Jin, Y.-N. Yeh, and X.-D. Zhang.
\newblock Laplacian coefficient, matching polynomial and incidence energy of
  trees with described maximum degree.
\newblock {\em J. Comb. Optim.}, 31:1345--1372, 2016.

\bibitem{Jooy}
M.~R. Jooyandeh, D.~Kiani, and M.~Mirzakhah.
\newblock Incidence energy of a graph.
\newblock {\em MATCH Commun. Math. Comput. Chem.}, 62:561--572, 2009.

\bibitem{kel}
A.~K. Kel'mans.
\newblock On properties of the characteristic polynomial of a graph.
\newblock {\em Cybernetics in the Service of Communism [In Russian]}, 4:27--41,
  1967.

\bibitem{kirk}
R.~Kirk and H.~Wang.
\newblock Largest number of subtrees of trees with a given maximum degree.
\newblock {\em SIAM J. Discrete Math.}, 22(3):985--995, 2008.

\bibitem{Shuchao}
S.~Li and Y.~Song.
\newblock On the spectral moment of trees with given degree sequences.
\newblock {\em arXiv:math/1209.2188}, 2012.

\bibitem{steiner}
X.~Li, Y.~Mao, and I.~Gutman.
\newblock The {S}teiner {W}iener index of a graph.
\newblock {\em Disc. Math. Graph Theory}, 36(2):455--465, 2016.

\bibitem{Hong}
H.~Lin.
\newblock On the {W}iener {I}ndex of {T}rees with {G}iven {N}umber of
  {B}ranching {V}ertices.
\newblock {\em MATCH Commun. Math. Comput. Chem.}, 72:301--310, 2014.

\bibitem{Liu}
J.~Liu and B.~Liu.
\newblock A {L}aplacian-energy like invariant of a graph.
\newblock {\em MATCH Commun. Math. Comput. Chem.}, 59:355--372, 2008.

\bibitem{Prodinger101}
H.~Prodinger and R.~F. Tichy.
\newblock Fibonacci numbers of graphs.
\newblock {\em Fibonacci Quart}, 20:16--21, 1982.

\bibitem{Schmuck}
N.~S. Schmuck, S.~G. Wagner, and H.~Wang.
\newblock Greedy {T}rees, {C}aterpillars, and {W}iener-type {G}raph
  {I}nvariants.
\newblock {\em MATCH Commun. Math. Comput. Chem.}, 68:273--292, 2012.

\bibitem{Szek2}
L.~A. Sz{\'e}kely and H.~Wang.
\newblock On subtrees of trees.
\newblock {\em Adv. Appl. Math.}, 34(1):138--155, 2005.

\bibitem{szekely}
L.~A. Sz{\'e}kely, H.~Wang, and T.~Wu.
\newblock The sum of the distances between the leaves of a tree and the
  semi-regular property.
\newblock {\em Discrete Math.}, 311(13):1197--1203, 2011.

\bibitem{Wagnerhar}
S.~Wagner, H.~Wang, and X.-D. Zhang.
\newblock Distance-based graph invariants of trees and the {H}arary index.
\newblock {\em Filomat}, 27(1):41--50, 2013.

\bibitem{WangWi}
H.~Wang.
\newblock The extremal values of the {W}iener index of a tree with given degree
  sequence.
\newblock {\em Discrete Appl. Math.}, 156:2647--2654, 2009.

\bibitem{wang10}
H.~Wang.
\newblock On trees with given number of pendant edges and their {W}iener
  indices.
\newblock {\em Adv. Appl. Math. Sci.}, 2(1):167--175, 2010.

\bibitem{wang14}
H.~Wang.
\newblock Functions on adjacent vertex degrees of trees with given degree
  sequence.
\newblock {\em Cent. Eur. J. Math.}, 12(11):1656--1663, 2014.

\bibitem{yeyuan}
L.~Ye and X.~Yuan.
\newblock On the minimal energy of trees with a given number of pendent
  vertices.
\newblock {\em MATCH Commun. Math. Comput. Chem.}, 57(1):193--201, 2007.

\bibitem{yulv}
A.~Yu and X.~Lv.
\newblock Minimum energy on trees with {$k$} pendent vertices.
\newblock {\em Linear Algebra Appl.}, 418(2-3):625--633, 2006.

\bibitem{Yu}
A.~Yu and X.~Lv.
\newblock The {M}errifield--{S}immons indices and {H}osoya indices of trees
  with $k$ pendant vertices.
\newblock {\em J. Math. Chem}, 41:33--43, 2007.

\bibitem{zzwz}
J.~Zhang, G.-J. Zhang, H.~Wang, and X.-D. Zhang.
\newblock Extremal trees with respect to the {S}teiner {W}iener index.
\newblock {\em Discrete Math. Algorithms Appl.}, 11(6):article 1950067, 16pp.,
  2019.

\bibitem{Zhang2}
X.-D. Zhang.
\newblock The {L}aplacian spectral radii of trees with given degree sequences.
\newblock {\em Discrete Math.}, 308(15):3143--3150, 2008.

\bibitem{Zhang}
X.-D. Zhang, Q.-Y. Xiang, L.-Q. Xu, and R.-Y. Pan.
\newblock The {W}iener index of trees with given degree sequences.
\newblock {\em MATCH Commun. Math. Comput. Chem}, 60:623--644, 2008.

\bibitem{Zhang1}
X.-M. Zhang, X.-D. Zhang, D.~Gray, and H.~Wang.
\newblock The number of subtrees of trees with given degree sequence.
\newblock {\em J. Graph Theory}, 73(3):280--295, 2013.

\bibitem{Zhou}
B.~Zhou and I.~Gutman.
\newblock A connection between ordinary and {L}aplacian spectra of bipartite
  graphs.
\newblock {\em Linear Multilinear Algebra}, 56:305--310, 2008.

\end{thebibliography}

\end{document}